\documentclass[a4paper,10pt]{amsart}

\usepackage[utf8]{inputenc}
\usepackage[T1]{fontenc}

\usepackage{amsmath,amssymb,amsthm,amsfonts}

\usepackage[
pagebackref = true,
colorlinks = true,
pdfstartview = FitH,
pdfpagemode = UseNone,
pdfauthor    = {Giacomo Cherubini},
pdftitle     = {Title},
pdfkeywords  = {keywords},
pdfcreator   = {Pdflatex},
linktoc = none,
bookmarks = false
]
{hyperref}

\usepackage[margin=3.5cm]{geometry}
\numberwithin{equation}{section}
\usepackage{mathrsfs}
\usepackage{euscript}
\usepackage{csquotes}
\usepackage{tikz}
\usetikzlibrary{arrows.meta}

\usepackage{subfig}
\usepackage{caption}

\def\eps{\epsilon}

\def\HH{\mathbb{H}}
\def\NN{\mathbb{N}}
\def\ZZ{\mathbb{Z}}
\def\RR{\mathbb{R}}
\def\CC{\mathbb{C}}
\def\QQ{\mathbb{Q}}

\def\Ccal{\mathcal{C}}
\def\Lcal{\mathcal{L}}
\def\Ncal{\mathcal{N}}
\def\Ocal{\mathcal{O}}
\def\Pcal{\mathcal{P}}

\def\Scal{\mathcal{S}}

\def\PSL{\mathrm{PSL}}
\def\Radius{\EuScript{R}}

\newcommand{\leg}[2]{\genfrac{(}{)}{}{}{#1}{#2}}
\newcommand{\zqbar}{\overline{z_q\!\!}\;}
\newcommand{\IndexSet}{\mathcal{I}}

\theoremstyle{plain}
\newtheorem{theorem}{Theorem}[section]
\newtheorem{proposition}[theorem]{Proposition}
\newtheorem{lemma}[theorem]{Lemma}
\newtheorem{corollary}[theorem]{Corollary}

\theoremstyle{definition}
\newtheorem{remark}[theorem]{Remark}

\title{Hyperbolic angles from Heegner points}

\author{Giacomo Cherubini}
\author{Alessandro Fazzari}

\address{
         Charles University,
         Faculty of Mathematics and Physics,
         Department of Algebra,
         Sokolov\-sk\'a 83, 18600 Praha~8,
         Czech Republic
        }

\email{
    cherubini@karlin.mff.cuni.cz
    }

\address{
         American Institute of Mathematics,
         600 East Brokaw Road,
         San Jose, CA 95112, US
		}
\address{
         Charles University,
         Faculty of Mathematics and Physics,
         Department of Algebra,
         Sokolov\-sk\'a 83, 18600 Praha~8,
         Czech Republic
        }
        
\email{
    fazzari@aimath.org\\
	fazzari@karlin.mff.cuni.cz
    }


\date{\today}

\begin{document}

\begin{abstract}
We study lattice points on hyperbolic circles centred at Heegner points of class number one.
Our main result is that, on a density one subset of radii tending to infinity,
the angles of such points equidistribute on the unit circle.
To prove this, we establish a
connection between lattice points and algebraic integers
in the associated field having norm of a special form and satisfying a congruence condition.
As a by-product of this, we obtain an explicit formulation of the classical hyperbolic circle
problem as a shifted convolution sum for the function that counts the number of algebraic integers with given norm.
Along the way, we also prove a lower bound for shifted $B$-numbers, which is done by sieve methods.
\end{abstract}
\maketitle

\section{Introduction}

The distribution of lattice points in the hyperbolic plane
and hyperbolic $n$-dimensional space is a well-studied subject,
with contributions dating back to Selberg \cite{selberg_equidistribution_1970}
and spanning until recent years.
The first type of problem one can investigate
asks whether the number of lattice points inside a hyperbolic ball
is proportional (with a prescribed constant) to the volume of the ball.
Asymptotics with explicit error terms are available in the literature and usually rely
on non-trivial inputs such as a \enquote{spectral gap} property of
the hyperbolic Laplacian (see e.g.~\cite{lax_asymptotic_1982}).

A feature of hyperbolic geometry is that the volume of a ball
is of the same order of magnitude as the measure of its boundary,
making ineffective a
geometric-type argument which would give the correct asymptotic
up to miscounting elements \enquote{near} the boundary.
For this reason, the behaviour of lattice points \emph{on} the boundary
deserves special attention.

In this paper, we work on the hyperbolic plane 
$\HH=\{x+iy,\, y>0\}$, equipped with the 
metric $ds^2=y^{-2}(dx^2+dy^2)$ and associated distance $\rho(z,w)$.
We focus on hyperbolic circles (the boundary of balls)
centred at Heegner points of class number one.
There are exactly nine imaginary quadratic fields $K$
with class number one; if $-q$ denotes their discriminant, then
\[
q = 3, 4, 7, 8, 11, 19, 43, 67, 163.
\]
For each value of $q$ as above we have a Heegner point $z_q\in\HH$, defined by
\begin{equation}\label{def:zq}
z_q:=\mu+i\lambda,
\quad\text{where}\quad
\mu=
\begin{cases}
0 & q=4,8,\\
\frac{1}{2} & \text{otherwise},
\end{cases}
\qquad
\lambda = \frac{\sqrt{q}}{2}.
\end{equation}
Note that the pair $\{1,z_q\}$ is a basis of
$\Ocal_K$, the ring of integers of $K$.
For instance, specialising to $q=4$ (i.e.~when the
underlying field is $\QQ(i)$), we have $z_q=i$.

The modular group $\Gamma=\PSL(2,\ZZ)$ acts on $\HH$
by linear fractional transformations and we will consider (for fixed $q$)
the set of lattice points $\{\gamma z_q,\gamma\in\Gamma\}$.
To a given lattice point, we associate an \enquote{arithmetic radius}
$\Radius(\gamma;z_q)$ given by the relation
\[
\Radius(\gamma;z_q) = 2\lambda^2\cosh(\rho(z_q,\gamma z_q))
\]
and show that this is always an integer or half an integer (see Section \ref{S2}).
Therefore, circles are parametrised by the set of arithmetic radii
\begin{equation}\label{intro:def:Nzq}
\Ncal_{z_q} := \{\Radius(\gamma;z_q),\;\gamma\in\Gamma\}.
\end{equation}
For $n\in\Ncal_{z_q}$, we also define
\begin{equation}\label{def:Gammazq}
\Gamma_{z_q,n} := \{\gamma\in\Gamma:\;\Radius(\gamma;z_q)=n\}.
\end{equation}

Let $w\in\HH$. We define the \emph{angle} of $w$ with respect to $z_q$ as follows:
first, we find the unique geodesic segment going from $z_q$ to $w$;
then, we consider the tangent line to this curve in~$z_q$;
finally, we take the angle it forms with the horizontal axis
(see Figure \ref{S2:fig} for a visual explanation of the construction).
When $w=\gamma z_q$, we will denote the angle by $\theta(\gamma)$.

Our main result is that, for a density one subset of $\Ncal_{z_q}$,
the lattice points $\{\gamma z_q:\gamma\in\Gamma_{z_q,n}\}$ become equidistributed,
by which we mean that the angles $\theta(\gamma)$
equidistribute on the unit circle $S^1$, as $n\to\infty$.
We prove this in a quantitative form, giving a bound on the discrepancy.

\begin{theorem}\label{intro:thm1}
Let $z_q$ be a Heegner point of class number one as defined in \eqref{def:zq}.
Define $\Ncal_{z_q}(x)=\{n\in\Ncal_{z_q}:n\leq x\}$.
Then $|\Ncal_{z_q}(x)|\asymp x/\log x$. Moreover,
for all but $o(|\Ncal_{z_q}(x)|)$ elements in $\Ncal_{z_q}(x)$,
we have $|\Gamma_{z_q,n}|\asymp (\log n)^{\log 2\pm o(1)}$ and
\[
\sup_{I\subseteq S^1} \Biggl|
\frac{\{\gamma \in\Gamma_{z_q,n}:\; \theta(\gamma)\in I\}}{|\Gamma_{z_q,n}|} - \frac{|I|}{2\pi}
\Biggr|
\ll_q
\frac{1}{|\Gamma_{z_q,n}|^{C-o(1)}},
\]
where $C=\log(\pi/2)/\log 2$.
\end{theorem}

When $q=4$, i.e.~$z_q=i$, Theorem \ref{intro:thm1} was proved by
Chatzakos--Lester--Kurlberg--Wigman \cite[Theorem 1.1]{chatzakos_distribution_2020}.
Our paper extends their results by showing an underlying structure for imaginary quadratic fields
other than $\QQ(i)$. As far as we know, very few papers study hyperbolic circles
centred at points other than $i$: in his thesis \cite[\S10]{steeples}, Steeples
considered base points $(i,i)$ and $(i,2i)$; later Malcolm \cite{malcolm}
looked at $(2i,2i)$.
See also Chamizo \cite[\S3]{chamizo_applications_1996} for arithmetic applications
of lattice point counting in the hyperbolic plane.
In a different direction, Petridis and Risager \cite{petridis_averaging_2018}
considered the classical hyperbolic circle problem
on average over Heegner points of discriminant $D$, as $D\to\infty$.
If one allows lattice points to lie in a full ball rather than only on its boundary,
then the angular equidistribution has been established in several works
\cite{boca,nicholls,risager-truelsen} and refined statistics have been studied too
\cite{boca-popa-zaharescu1,boca-popa-zaharescu2,kelmer-kontorovich,risager-sodergren,marklof-vinogradov}.

One key observation in the proof of \cite[Theorem 1.1]{chatzakos_distribution_2020}
is that lattice points on a given circle of (arithmetic) radius $n$ can be mapped
to integer points on a Euclidean circle of radius $\sqrt{n^2-4}$ satisfying a congruence condition.
This leads to study angles of complex points on such a Euclidean circle,
which is highly convenient due to the arithmetic nature of the coordinates.
An analogous situation occurs for all the points we consider.

\begin{proposition}\label{intro:prop1}
Let $K$ be an imaginary quadratic field of class number one,
with discriminant $-q$ and ring of integers $\Ocal_K$.
Let $z_q=\mu+i\lambda$ be as in \eqref{def:zq}
and let $\Ncal_{z_q},\Gamma_{z_q,n}$ be as in \eqref{intro:def:Nzq}
and \eqref{def:Gammazq}. For $n\in\Ncal_{z_q}$, we have
\[
\{\theta(\gamma):\gamma\in\Gamma_{z_q,n}\}
=
\biggl\{
\arg(x+iy)\Big|
\begin{array}{c}
y+ix\in\Ocal_K,\; N(y+ix)=n^2-4\lambda^4,\\
\; 2y\equiv 2n\!\!\!\!\pmod{q}
\end{array}
\biggr\}.
\]
\end{proposition}

We expect that Proposition \ref{intro:prop1}
remains valid if one works with points of the form $z_q=ki$, $k\in\NN$,
but one may have to exclude certain values (arithmetic progressions) for $n$.
For instance, if $z_q=2i$ and $n=10$, the equality in the proposition fails,
as the left-hand side is empty while the right-hand side is not
(see Remark \ref{S2:rmk}). This was already observed by Malcolm \cite[Theorem 4.5]{malcolm}.

In relation to the classical hyperbolic circle problem we mention that,
as a by-product of the proof of Proposition \ref{intro:prop1}, we obtain that
the cardinality of $\Gamma_{z_q,n}$ can be expressed in terms of the function $r_K(n)$,
the number of algebraic integers in $\Ocal_K$ with norm $n$.

\begin{proposition}\label{intro:prop2}
Let $K$ be an imaginary quadratic field of class number one and let $\Gamma_{z_q,n}$
be defined as in \eqref{def:Gammazq}. Then
\[
|\Gamma_{z_q,n}| = c_n r_K(n-2\lambda^2) r_K(n+2\lambda^2),
\]
where
\[
c_n =
\begin{cases}
1/2 & \text{$q$ even and $2|n$ or $q$ odd and $q|2n$}\\
1/4 & \text{otherwise}.
\end{cases}
\]
\end{proposition}

On summing over $n$, we obtain a shifted convolution sum
which is familiar in the case $z_q=i$ \cite[(1.14)]{friedlander_hyperbolic_2009},
but has not appeared before for points off the imaginary axis.
To give an example, we write the sum explicitly in the case $q=3$
(with the standard asymptotic due to Selberg,
see e.g.~\cite[Theorem 12.1]{iwaniec_methods_2002}).

\begin{corollary}\label{intro:cor}
Let $q=3$, $z_q=\frac{1+i\sqrt{3}}{2}$ and $K=\QQ(z_q)$.
Let $r_K(n)$ denote the number of algebraic integers in $\Ocal_K$
with norm $n$. Then
\[
\#\{\gamma\in\Gamma:\cosh(\rho(z_q,\gamma z_q))\leq x\}
=
\sum_{n\leq \frac{3x-1}{2}} c_n r_K(n+2)r_K(n-1) = 6x + O(x^{2/3}).
\]
where $c_n=1/2$ if $n\equiv 1\!\!\pmod{3}$ and $c_n=1/4$ otherwise.
\end{corollary}

Another ingredient in the proof of Theorem \ref{intro:thm1}
is a lower bound for $B$-numbers for number fields other than $\QQ(i)$.
In the literature, integers that can be represented
as a sum of two squares are sometimes called $B$-numbers
\cite{indlekofer_scharfe_1974,nowak_distribution_2005}
and the corresponding indicator function is denoted by $b(\cdot)$,
see also \cite[\S14.3]{friedlander_opera_2010}.
By extension, we say that an integer $n$ is a $B$-number for the field $K$
if $b_K(n)=1$, where
\[
b_K(n)=
\begin{cases}
1 & \text{if $n$ is the norm of an ideal in $\Ocal_K$},\\
0 & \text{otherwise}.
\end{cases}
\]
Restricting to number fields with class number one,
we can (and will) simply talk about algebraic integers
rather than ideals.
With the above notation for $b_K$, we prove the following.

\begin{theorem}\label{intro:thm2}
Let $K$ be an imaginary quadratic field of class number one
and let $h\in\ZZ$. Then
\[
\sum_{n\leq x} b_K(n)b_K(n+h) \gg_{K,h} \frac{x}{\log x}.
\]
\end{theorem}

When $h=0$, since $b_K^2=b_K$, Theorem \ref{intro:thm2} follows from Bernays' work
\cite[p.91--92]{bernays_uber_1912} (see~also Odoni \cite{odoni_norms_1975}).
When $h\neq 0$ and $K=\QQ(i)$, the result was proved by Hooley \cite{hooley_intervals_1974}
and independently by Indlekofer \cite{indlekofer_scharfe_1974}.
For different fields we found no reference, although
Nowak \cite{nowak_distribution_2005}
proved an upper bound of the correct order of magnitude
(in fact he even allows several linear factors).
The proof of Theorem \ref{intro:thm2} uses sieve methods and follows
the lines of \cite{indlekofer_scharfe_1974}.

The paper is organised as follows: in Section \ref{S2} we describe the geometric
settings for the problem and prove Propositions \ref{intro:prop1} and \ref{intro:prop2};
in Section \ref{S3} we prove Theorem \ref{intro:thm2}; finally, in Section~\ref{S4}
we combine the results of the previous sections to obtain the proof
of Theorem~\ref{intro:thm1}.

\subsection*{Acknowledgements}
We thank D.~Chatzakos, S.~Lester, S.~Pujahari for their comments
and Y.~Petridis for sharing the references \cite{malcolm} and \cite{steeples}.
This work was supported by the Czech Science Foundation GACR, grant 21-00420M,
the project PRIMUS/20/SCI/002 from Charles University,
Charles University Research Centre program UNCE/SCI/022.
The second named author is supported by the FRG grant DMS 1854398.

\section{Geometric considerations}\label{S2}

In this section we discuss several geometric aspects of the problem
and prepare the ground for the proof of Theorem \ref{intro:thm1},
which will be given in Section \ref{S4}. Along the way,
we prove Proposition \ref{intro:prop1} and Proposition \ref{intro:prop2}.

The hyperbolic distance $\rho(z,w)$ between points $z,w\in\HH$
can be expressed in terms of
the Euclidean one by the identity (see e.g.~\cite[Theorem 1.2.6]{katok_fuchsian_1992})
\begin{equation}\label{2012:eq001}
\cosh \rho(z,w) = 1 + \frac{|z-w|^2}{2\Im(z)\Im(w)},
\end{equation}
where $|z-w|$ is the usual absolute value in $\CC$.
Let $z=\mu+i\lambda\in\HH$ and define the function
\[
\Radius(a,b,c,d;z) := a^2|z|^2 + b^2 + c^2|z|^4 + d^2|z|^2 + 2\mu((a-d)(b-c|z|^2) - \mu(ad+bc)).
\]
When $a,b,c,d$ are the entries of a matrix in $\PSL(2,\RR)$,
then a calculation using \eqref{2012:eq001} shows that the distance $\rho(z,\gamma z)$
is closely related to $\Radius(a,b,c,d;z)$.

\begin{lemma}\label{lemma:S2.1}
Let $z=\mu+i\lambda\in\HH$ and let $\gamma=(\begin{smallmatrix}a&b\\c&d\end{smallmatrix})$
be a real matrix with $ad-bc=1$. Then
\[
\cosh\rho(z,\gamma z) = \frac{\Radius(a,b,c,d;z)}{2\lambda^2}.
\]
\end{lemma}

\begin{proof}
First we have, by directly expanding and multiplying out all the terms,
\[
\begin{split}
|\gamma z-z|^2
&=
\Bigl|\frac{az+b-cz^2-dz}{cz+d}\Bigr|^2
\\
&=
\frac{((a-d)z+b-cz^2)((a-d)\overline{z}+b-c\overline{z}^2)}{|cz+d|^2}
=
\frac{\Radius(a,b,c,d;z)-2\lambda^2}{|cz+d|^2}.
\end{split}
\]
Since $\Im(z)=\lambda$ and the determinant condition implies $\Im(\gamma z)=\lambda|cz+d|^{-2}$,
the result follows from \eqref{2012:eq001}.
\end{proof}

\emph{Notation.} From now on, we will simply write
$\Radius(\gamma;z)$ instead of $\Radius(a,b,c,d;z)$.
In the rest of this section, $a,b,c,d$ will always denote
the entries of a matrix $\gamma\in\PSL(2,\RR)$.

Next, we look at angles of lattice points. As mentioned in the introduction,
they are defined as angles between geodesics.
However, in order to visualize and study them in a more comfortable way,
it is convenient to map the hyperbolic plane to the unit disc model
(see Figure \ref{S2:fig}).
Thus, for any given $z\in\HH$, we define the map
\begin{equation}\label{def:f}
f(w) = \frac{i(w-z)\,}{w-\overline{z}}.
\end{equation}
Clearly, $f$ maps $z$ to the origin. Also,
$f$ maps the real line to the unit circle
and the hyperbolic plane to its interior.
Regarding the action of $\PSL(2,\RR)$,
let us write again $z=\mu+i\lambda$
and let $\gamma\in\PSL(2,\RR)$ with $\Radius(\gamma;z)=n$.
In Lemma \ref{lemma:S2.2} we will show that, up to a constant,
$f$~maps the point $\gamma z$ to the point $x_\gamma+iy_\gamma$, where
\begin{equation}\label{def:xy}
\begin{cases}
x_\gamma = 2\lambda((\mu a+b)(\mu c+d)+\lambda^2ac-\mu((\mu c+d)^2+\lambda^2c^2)),\\
y_\gamma = n-2\lambda^2((\mu c+d)^2+\lambda^2c^2).
\end{cases}
\end{equation}
Since $f$ is a holomorphic diffeomorphism, hence conformal, angles
of lattice points can be studied by simply looking at 
angles of $x_\gamma+iy_\gamma$ as $\gamma$ varies in $\PSL(2,\ZZ)$.
%
%
%
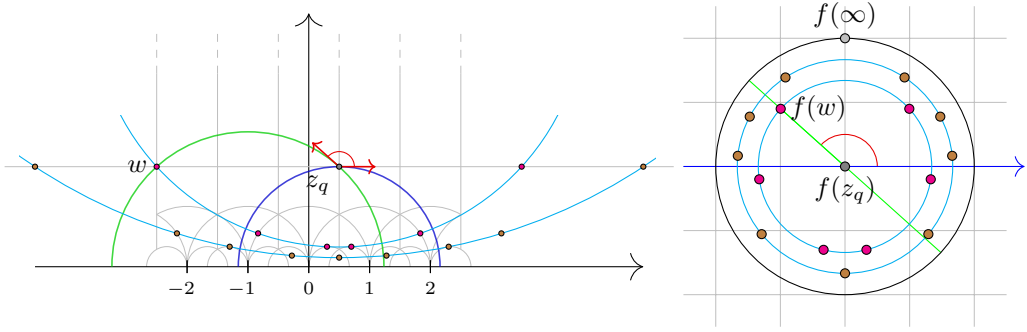
\begin{figure}[!h]
   \subfloat{%
      \begin{tikzpicture}[scale=0.8]
   
       \def\TopEdge{3.2}
       \def\TopTopEdge{3.9}
   	\begin{scope}[gray!50, very thin]
      	    \foreach \x in {-2.5,-1.5,-0.5,0.5,1.5,2.5}
      	        {
                 \draw[dashed] (\x,\TopEdge) -- (\x, \TopTopEdge);
      	        \draw (\x,0.288675) -- (\x,\TopEdge);
      	        }
      	    \foreach \x in {-1.5,-0.5,0.5,1.5,2.5}
      	        \draw (\x,{sin(60)}) arc [start angle = 60, end angle = 120, radius = 1];
      	    \foreach \x in {-2,-1,0,1,2}
      	        {
      	        \draw (\x,0) arc [start angle = 0, end angle = 60, radius = 1];
      	        \draw (\x,0) arc [start angle = 180, end angle = 120, radius = 1];
   	           }
   	       \foreach \x in {-2.333,-1.666,-1.333,-0.666,-0.333,0.333,0.666,1.333,1.666,2.333}
   	           {
   	           \clip (-3,0) rectangle (3,0.7);
   	           \draw (\x,0) circle [radius = 0.3333];
   	           }
       \end{scope}
   
      \begin{scope}[arrows={->[scale=2]}]
          \draw (-4.5,0) -- (5.5,0);
   	    \draw (0,0) -- (0,4.2);
   	\end{scope}
   
   	\foreach \x in {-2,-1,0,1,2}
   	    \draw (\x,3pt) -- (\x,-3pt);
   	\foreach \x/\xtext in {-2/-2\;\;,-1/-1\;\;,0,1,2}
   	    \draw (\x cm,3pt)--(\x cm,-3pt) node[below] {$\scriptstyle\xtext$};

      \def\myLambda{1.65831}
      \draw[very thin,gray!50] (-5,\myLambda)--(6,\myLambda);
   
      \draw[semithick,blue!80!black,opacity=0.75] ({\myLambda+0.5},0) arc [start angle = 0, end angle = 180, radius = \myLambda];
   
      \def\myRadius{2.23606}
      \draw[semithick,green!80!black,opacity=0.75] (-1+\myRadius,0) arc [start angle = 0, end angle = 180, radius = \myRadius];

      \coordinate (zq) at (0.5,\myLambda);
      \draw[arrows={->[scale=0.95]},red!90!black,semithick] (zq)--(1.1,\myLambda);
      \draw[arrows={->[scale=0.95]},red!90!black,semithick] (zq)--({0.5-0.6*cos(42.13)},{\myLambda+sin(42.13)*0.6});
      \draw[fill=gray] (zq) circle [radius=1.2pt] node[below left] {$z_q$};
   
      \draw[red!90!black] (0.75,\myLambda) arc [start angle = 0, end angle = 138, radius = 0.25];

      \clip (-4.75,-1) rectangle (5.7,2.5);
      \def\myRho{1.6244665}
      \draw[cyan] (0.5,{0.5*\myLambda*(exp(\myRho)+exp(-\myRho))}) circle [radius = 0.5*\myLambda*(exp(\myRho)-exp(-\myRho))];
   
      \coordinate (zq1) at (0.3,0.331662479);
      \coordinate (zq2) at (-2.5,1.65831239);
      \coordinate (zq3) at (-0.8333,0.55277);
      \coordinate (zq4) at (0.7,0.331662479);
      \coordinate (zq5) at (3.5,1.65831239);
      \coordinate (zq6) at (1.83333,0.55277);
      \draw[fill=magenta] (zq1) circle [radius=1.2pt];
      \draw[fill=magenta] (zq2) circle [radius=1.2pt] node[left] {$w$};
      \draw[fill=magenta] (zq3) circle [radius=1.2pt];
      \draw[fill=magenta] (zq4) circle [radius=1.2pt];
      \draw[fill=magenta] (zq5) circle [radius=1.2pt];
      \draw[fill=magenta] (zq6) circle [radius=1.2pt];
   
      \def\myRhoTwo{2.39789527}
      \draw[cyan] (0.5,{0.5*\myLambda*(exp(\myRhoTwo)+exp(-\myRhoTwo))}) circle [radius = 0.5*\myLambda*(exp(\myRhoTwo)-exp(-\myRhoTwo))];
   
      \coordinate (zq10) at (-0.2777,0.184256);
      \coordinate (zq11) at (-4.5,\myLambda);
      \coordinate (zq12) at (1.2777,0.184256);
      \coordinate (zq13) at (0.5,0.150755);
      \coordinate (zq14) at (-1.3,0.331662);
      \coordinate (zq15) at (5.5,\myLambda);
      \coordinate (zq16) at (2.3,0.331662);
      \coordinate (zq17) at (-2.1666,0.55277);
      \coordinate (zq18) at (3.1666,0.55277);
      \draw[fill=brown] (zq10) circle [radius=1.2pt];
      \draw[fill=brown] (zq11) circle [radius=1.2pt];
      \draw[fill=brown] (zq12) circle [radius=1.2pt];
      \draw[fill=brown] (zq13) circle [radius=1.2pt];
      \draw[fill=brown] (zq14) circle [radius=1.2pt];
      \draw[fill=brown] (zq15) circle [radius=1.2pt];
      \draw[fill=brown] (zq16) circle [radius=1.2pt];
      \draw[fill=brown] (zq17) circle [radius=1.2pt];
      \draw[fill=brown] (zq18) circle [radius=1.2pt];
   
   \end{tikzpicture}
   }
   \subfloat{%
   \begin{tikzpicture}[scale=1.7]
   
      \draw[very thin,gray!50,step=0.5] (-1.25,-1.25) grid (1.25,1.25);
   
      \draw[red!90!black] (0.25,0) arc [start angle = 0,end angle = 138, radius=0.25];
      
      \draw[blue,arrows={->[scale=1.5]}] (-1.25,0)--(1.39,0);

      \draw[green] ({-cos(42.13)},{sin(42.13)})--({cos(42.13)},{-sin(42.13)});
      \draw[green] (0,0)--(-0.4974937,0.45);
   
      \draw (0,0) circle [radius=1];
      \draw[fill=gray] (0,0) circle [radius=1pt] node[below] {$f(z_q)$};
      
      \draw[cyan] (0,0) circle [radius=0.67082039];
      \draw[cyan] (0,0) circle [radius=0.8333];
   
      \draw[fill=gray!50] (0,1) circle [radius=1pt] node[above] {$f(\infty)$};
   
      \draw[fill=magenta] (-0.165831,-0.65) circle [radius=1pt];
      \draw[fill=magenta] (-0.4974937,0.45) circle [radius=1pt] node[right] {$f(w)$};
      \draw[fill=magenta] (-0.663324958,-0.1) circle [radius=1pt];
      \draw[fill=magenta] (0.165831,-0.65) circle [radius=1pt];
      \draw[fill=magenta] (0.4974937,0.45) circle [radius=1pt];
      \draw[fill=magenta] (0.663324958,-0.1) circle [radius=1pt];

      \draw[fill=brown] (-0.644899,-0.52777) circle [radius=1pt];
      \draw[fill=brown] (-0.46064233,0.69444) circle [radius=1pt];
      \draw[fill=brown] (0.644899,-0.52777) circle [radius=1pt];
      \draw[fill=brown] (0,-0.8333) circle [radius=1pt];
      \draw[fill=brown] (-0.829156,0.08333) circle [radius=1pt];
      \draw[fill=brown] (0.46064233,0.69444) circle [radius=1pt];
      \draw[fill=brown] (0.829156,0.08333) circle [radius=1pt];
      \draw[fill=brown] (-0.7370277,0.3888) circle [radius=1pt];
      \draw[fill=brown] (0.7370277,0.3888) circle [radius=1pt];
   
   \end{tikzpicture}
   }
\captionsetup{size=footnotesize}
\caption{Lattice points for the Heegner point $z_{q}$ associated to $q=11$ and $\Radius(\gamma;z_{q})$ equal to $\tfrac{29}{2}$ (pink) and $\tfrac{61}{2}$ (brown).}\label{S2:fig}
\end{figure}

\begin{lemma}\label{lemma:S2.2}
Let $z=\mu+i\lambda\in\HH$ and $\gamma\in\PSL(2,\RR)$.
Let $f$ be the map defined in \eqref{def:f}.
Set $n=\Radius(\gamma;z)=2\lambda^2\cosh\rho(z,\gamma z)$.
Then
\[
f(\gamma z) = \frac{1}{n+2\lambda^2}(x_\gamma + iy_\gamma),
\]
where $x_\gamma,y_\gamma$ are as in \eqref{def:xy}.
Furthermore, we have
$x_\gamma^2 + y_\gamma^2 = n^2-4\lambda^4$.
\end{lemma}

\begin{proof}
By the identities $|w-\overline{z}|^2=|w-z|^2+4\Im(w)\Im(z)$
and $\Im(\gamma z)=\lambda|cz+d|^{-2}$ we deduce
\[
|\gamma z-\overline{z}|^2 = \frac{2\lambda^2(\cosh\rho(z,\gamma z)+1))}{|cz+d|^2} = \frac{(n+2\lambda^2)}{|cz+d|^2}.
\]
Therefore,
\[
f(\gamma z)
=
i\,\frac{|cz+d|^2}{n+2\lambda^2}(\gamma z-z)(\overline{\gamma z} - z).
\]
Expanding the product gives
\[
\begin{split}
(n+2\lambda^2)f(\gamma z)
={}&
2\lambda((\mu a+b)(\mu c+d)+\lambda^2ac-\mu((\mu c+d)^2+\lambda^2c^2))
\\
&\phantom{xxxxxxx}+
i\bigl( |z|^2a^2 + b^2 + \mu^2|z|^2c^2 + \mu^2d^2 -\lambda^2|cz+d|^2
\\
&\phantom{xxxxxxxxxxxxxxx}+
2\mu( a(b-|z|^2c) - d(b-\mu^2c) - \mu(ad+bc) )\bigr).
\end{split}
\]
We recognize the real part as $x_\gamma$. As for the imaginary part, we add and subtract $\lambda^2|cz+d|^2$
and obtain
\[
|z|^2a^2+b^2+|z|^4c^2+|z|^2d^2+2\mu((a-d)(b-|z|^2c)-\mu(ad+bc)) - 2\lambda|cz+d|^2.
\]
The sum of all the terms but the last one gives $\Radius(a,b,c,d;z)$ 
and so we obtain $y_\gamma$. Now let us show the last part of the lemma,
i.e., $x_\gamma^2+y_\gamma^2=n^2-4\lambda^4$.
Proceeding as in the beginning of the proof, we can write
\[
|f(w)|^2
=
\frac{|w-z|^2}{|w-\overline{z}|^2}
=
\frac{|w-z|^2}{2\Im(w)\Im(z)} \left(\frac{|w-z|^2}{2\Im(w)\Im(z)}+2\right)^{\!-1}
=
\frac{\cosh\rho(z,w)-1}{\cosh\rho(z,w)+1}.
\]
Setting $w=\gamma z$ and recalling that $n=2\lambda^2\cosh\rho(z,\gamma z)$ gives
\[
|f(\gamma z)|^2 = \frac{n-2\lambda^2}{n+2\lambda^2}.
\]
Since $x_\gamma^2+y_\gamma^2=(n+2\lambda^2)^2|f(\gamma z)|^2$, we obtain the claim.
\end{proof}

\subsection{Heegner points}

Let us specialise to the Heegner points $z_q$ defined in \eqref{def:zq}.
If $\gamma\in\Gamma$, then $\Radius(\gamma;z_q)$
is an integer (if~$q$ is even) or half an odd integer
(if $q$ is odd). Indeed, when $q$ is even we have $\mu=0$ and thus
\[
\Radius(\gamma;z_q)
=
a^2|z_q|^2 + b^2 + c^2|z_q|^4 + d^2|z_q|^2,
\]
which is an integer since $|z_q|^2\in\NN$.
When $q$ is odd, and so $\mu=1/2$,
recalling also the identity $ad-bc=1$,
we obtain
\[
\begin{split}
2\Radius(\gamma;z_q)
={}&
2a^2|z_q|^2 + 2b^2 + 2c^2|z_q|^4 + 2d^2|z_q|^2 
\\
&
+4\mu(a-d)(b+c|z_q|^2) - 4\mu^2(ad+bc)
\equiv ad+bc \equiv 1 \pmod{2},
\end{split}
\]
which shows that $\Radius(\gamma;z_q)$ is half an odd integer.
Combining the two cases, we can summarize by saying that
the set of arithmetic radii
$\Ncal_{z_q}$ is a subset of $(1-\mu)\NN$.

For $n\in\Ncal_{z_q}$, we wish to study angles attached to
lattice points $\gamma z_q$ with $\gamma\in\Gamma_{z_q,n}$.
The first step consists in using Lemma \ref{lemma:S2.2} to map
hyperbolic circles to Euclidean ones centred at the origin.
To do this, we use the map
\begin{equation}\label{2401:eq001}
\phi:\gamma\longmapsto x_\gamma + iy_\gamma = (n+2\lambda^2)f(\gamma z_q),
\end{equation}
where $f,x_\gamma$ and $y_\gamma$ are defined in \eqref{def:f} and \eqref{def:xy}.
Note that, as a map from matrices to points,
the function $\phi$ is not always injective.

\begin{lemma}\label{lemma:S2.3}
Let $z_q$ be a Heegner point of class number one, as defined in \eqref{def:zq},
and let $\phi$ be the map defined in \eqref{2401:eq001}.
Then every point in the image of $\phi$
appears a number of times equal to the cardinality of the stabilizer of $z_q$ in $\PSL(2,\ZZ)$.
Equivalently, it appears $\frac{1}{2}|\mathcal{O}_K^\times|$ times, where
$|\mathcal{O}_K^\times|$ is the number of units in $\mathcal{O}_K$.
\end{lemma}

\begin{proof}
Since $f$ is bijective, we have $f(\gamma_1z_q)=f(\gamma_2z_q)$
if and only if $\gamma_1z_q=\gamma_2z_q$, that is, $\gamma_2^{-1}\gamma_1$
stabilizes $z_q$. Apart from the identity matrix,
the imaginary unit $i$ is fixed by the matrix
$(\begin{smallmatrix}0&-1\\1&0\end{smallmatrix})$,
while
the point $z_q=\frac{1+i\sqrt{3}}{2}$ is fixed by
$(\begin{smallmatrix}1&-1\\1&0\end{smallmatrix})$
and
$(\begin{smallmatrix}0&1\\-1&1\end{smallmatrix})$.
All other points have trivial stabilizer.
The last sentence follows from the fact that $|\Ocal_K^\times|=4,6$
for $K=\QQ(i)$ and $\QQ(\frac{1+i\sqrt{3}}{2})$ respectively, whereas $|\Ocal_K^\times|=2$
in all other cases.
\end{proof}

We now describe a connection between $x_\gamma+iy_\gamma$, which has norm $n^2-4\lambda^4$
by Lemma \ref{lemma:S2.2}, and two algebraic integers in $\Ocal_K$ of norm $n\pm 2\lambda^2$,
from which we can recover $\gamma$.
Such a connection will be crucial for the proof of Proposition \ref{intro:prop1}
and Proposition \ref{intro:prop2}.

Starting from $\gamma=(\begin{smallmatrix}a&b\\c&d\end{smallmatrix})\in\PSL(2,\ZZ)$
and writing $z_q=\mu+i\lambda$ as in \eqref{def:zq}, we define four integers $r,u,s,t$ by:
\begin{equation}\label{def:rust}
\left\{\begin{array}{l}
r = a + d,\\
u = b - |z_q|^2c - 2\mu d
\end{array}\right.
\qquad
\left\{\begin{array}{l}
s = a-d-2\mu c,\\
t = b + |z_q|^2c.
\end{array}\right.
\end{equation}
A direct computation shows that
\[
\begin{gathered}
u^2+2\mu ur+|z_q|^2r^2 = \Radius(\gamma;z_q) + 2\lambda^2(ad-bc),\\
t^2+2\mu st+|z_q|^2s^2 = \Radius(\gamma;z_q)-2\lambda^2(ad-bc).
\end{gathered}
\]
The quantities on the left are norms
of algebraic integers in $\Ocal_K$, respectively of $u+rz_q$ and $t+sz_q$.
Upon setting $n=\Radius(\gamma;z_q)$ and recalling that $ad-bc=1$, we have
\[
N(u+rz_q) = n+2\lambda^2,
\quad
N(t+sz_q)=n-2\lambda^2.
\]
The definition of $r,u,s,t$ can be written in matrix form as
\[
\begin{pmatrix}r\\u\\s\\t\end{pmatrix}
=
\begin{pmatrix}
1 & 0 & 0 & 1\\
0 & 1 & -|z_q|^2 & -2\mu\\
1 & 0 & -2\mu & -1\\
0 & 1 & |z_q|^2 & 0
\end{pmatrix}
\begin{pmatrix}a\\b\\c\\d\end{pmatrix}.
\]
The determinant of the above matrix is $4(|z_q|^2-\mu^2)=4\lambda^2=q$,
so the transformations can be inverted and give
\begin{equation}\label{2501:eq002}
\begin{pmatrix}a\\b\\c\\d\end{pmatrix}
=
\frac{1}{q}
\begin{pmatrix}
(2|z_q|^2-4\mu^2) & -2\mu & 2|z_q|^2 & 2\mu\\
2\mu|z_q|^2 & 2|z_q|^2 & -2\mu|z_q|^2 & (2|z_q|^2-4\mu^2)\\
-2\mu & -2 & 2\mu & 2\\
2|z_q|^2 & 2\mu & -2|z_q|^2 & -2\mu
\end{pmatrix}
\begin{pmatrix}r\\u\\s\\t\end{pmatrix}.
\end{equation}
Denote by $M$ the above $4\times 4$ matrix
and use the shorthand $\mathbf{v}=(r,u,s,t)^t$.
Clearly, the numbers $a,b,c,d$ obtained from $\mathbf{v}$
in this way are integers if and only if
\begin{equation}\label{2501:eq001}
M \mathbf{v}\equiv 0\pmod{q}.
\end{equation}
Therefore, the above construction shows that
$\Gamma_{z_q,n}$ is in bijection with the set
\begin{equation}\label{def:Czqn}
\Ccal_{z_q,n}=
\left\{
u+rz_q,\,t+sz_q\in\Ocal_K \;\Bigg|\;
\begin{array}{c}
N(u+rz_q)=n+2\lambda^2,\\
N(t+sz_q)=n-2\lambda^2\rule{0pt}{10pt}\\
\text{ and }\eqref{2501:eq001} \text{ holds}\rule{0pt}{10pt}
\end{array}
\right\}/\{\pm 1\}.
\end{equation}
The quotient accounts for the fact that $(r,u,s,t)$ must be identified with $-(r,u,s,t)$
since $\gamma$ is identified with $-\gamma$ in $\PSL(2,\ZZ)$.

As an example, when $q=4$ then $\mu=0$, $\lambda=|z_q|=1$ and \eqref{2501:eq001} reduces to
\[
2r\pm 2s\equiv 2t\pm 2u \equiv 0 \pmod{4},
\]
which amounts to saying that $r$ and $s$ have the same parity, as do $u$ and $t$
(cf.~\cite[\S2.2]{chatzakos_distribution_2020}).
Similarly, when $q=8$ the system simplifies to
\begin{equation}\label{3103:eq001}
r\equiv s\!\!\!\pmod{2} \qquad u\equiv t\!\!\!\pmod{4}.
\end{equation}
When $q$ is odd, the system \eqref{2501:eq001}
gives four congruences involving all of $r,u,s$ and~$t$. However, they are in fact all equivalent.
To see this, note that $4|z_q|^2\equiv 1\!\!\pmod{q}$ and that $2$ is invertibile modulo $q$.
Multiplying the first row of $M$ by $2$ and reducing modulo $q$ gives the third row.
The same happens if we multiply the second and fourth rows
respectively by $-4$ and by $-2$. Therefore, when $q$ is odd,
the system \eqref{2501:eq001} is equivalent to the single condition
\begin{equation}\label{3103:eq004}
r+2u \equiv s+2t \pmod{q}.
\end{equation}

Furthermore, \eqref{def:Czqn} suggests a direct relation between the cardinality
$|\Gamma_{z_q,n}|$ and the product $r_K(n+2\lambda^2)r_K(n-2\lambda^2)$,
the total number of representations of $n\pm 2\lambda^2$ as norms in $\Ocal_K$. 
Making this precise yields Proposition \ref{intro:prop2}.

\begin{proof}[Proof of Proposition \ref{intro:prop2}]
The case $q=4$ is discussed e.g.~in \cite[\S2A]{chatzakos_distribution_2020}
or \cite[(1.12)--(1.14)]{friedlander_hyperbolic_2009},
so we only treat the cases $q=8$ and $q$ odd.
Our goal is to compare $\Ccal_{z_q,n}$ with the set
\begin{equation}\label{2903:eq003}
\{(\alpha,\beta)\in\Ocal_K^2:\;N(\alpha)=n+2\lambda^2,\;N(\beta)=n-2\lambda^2\}/\{\pm 1\}.
\end{equation}
\emph{Case $q=8$.} In this case, \eqref{2501:eq001} reduces to the two congruences in \eqref{3103:eq001}.
At the same time, the condition on the norms gives
\begin{equation}\label{2903:eq001}
u^2+2r^2 = n+4,\qquad t^2+2s^2=n-4.
\end{equation}
We distinguish according to the value of $n$ modulo $8$.
Assume $n$ is odd, so $n\equiv 5,7\!\!\pmod{8}$,
for otherwise \eqref{2903:eq001} has no solutions.
If $n\equiv 5\!\!\pmod{8}$, then
reducing \eqref{2903:eq001}$\pmod{8}$ we see
that $r$ and $s$ are even while $t$ and $u$ are odd.
In particular, the first condition in \eqref{3103:eq001} is always satisfied.
On the other hand, since the pair
$(\eps_1 u+rz_q,\eps_2 t+sz_q)$ is in \eqref{2903:eq003}
for any choice of signs $\eps_1,\eps_2\in\{\pm1\}$,
we deduce that only half the elements in \eqref{2903:eq003}
satisfy the second condition in \eqref{3103:eq001}.
In other words, in this case we have
\begin{equation}\label{3103:eq002}
|\Gamma_{z_q,n}| = |\Ccal_{z_q,n}| = c_n r_K(n+2\lambda^2)r_K(n-2\lambda^2)\quad\text{with}\quad c_n=\frac{1}{4}.
\end{equation}
When $n\equiv 7\!\!\pmod{8}$, again by looking at \eqref{2903:eq001}
modulo $8$, we deduce that $r,u,s,t$ are all odd.
Therefore, like before, the first condition in \eqref{3103:eq001}
is always satisfied while the second one
is satisfied by only half the elements in \eqref{2903:eq003},
which gives \eqref{3103:eq002} with $c_n=1/4$.

When $n$ is even, we claim that $c_n=1/2$ always. Indeed, if $n\equiv 2\!\!\pmod{8}$,
then reducing \eqref{2903:eq001} modulo $8$ gives $r,s$ odd
and $u\equiv t\equiv 2\!\!\pmod{4}$, so \eqref{3103:eq001} is automatically
satisfied. Similarly, when $n\equiv 6\!\!\pmod{8}$ we get $r,s$ odd and $u\equiv t\equiv 0\!\!\pmod{4}$.
If $n\equiv 4\!\!\pmod{8}$, then $r$ and $s$ are even and
$u\equiv t\equiv 0\!\!\pmod{4}$. Finally, $n\equiv 0\!\!\pmod{8}$ leads to $r,s$ even
and $u\equiv t\equiv 2\!\!\pmod{4}$.

\emph{Case $q$ odd.}
In this case the system \eqref{2501:eq001}
reduces to the single condition \eqref{3103:eq004}.
Let $u+rz_q\in\Ocal_K$ have norm $n+2\lambda^2$ and $t+sz_q\in\Ocal_K$
have norm $n-2\lambda^2$. In particular,
\[
(2u+r)^2 \equiv 4n \equiv (2t+s)^2 \pmod{q}.
\]
Therefore, $r+2u\equiv \pm (s+2t)$ modulo $q$.
If $q|2n$, then \eqref{3103:eq004} is automatically satisfied and $c_n=1/2$.
When $q\nmid 2n$, only half the points in \eqref{2903:eq003}
will satisfy \eqref{3103:eq004}, hence $c_n=1/4$ in this case.
\end{proof}

We return now to the points $x_\gamma+iy_\gamma$.
Lemma \ref{lemma:S2.3} gives a $\frac{1}{2}|\Ocal_K^\times|$-to-one map
$\phi:\Gamma_{z_q,n}\longmapsto L_n$, where
\begin{equation}\label{def:Ln}
L_n = \{(x_\gamma,y_\gamma) \text{ as in \eqref{def:xy}}, \;
\gamma=(\begin{smallmatrix}a&b\\c&d\end{smallmatrix})\in\PSL(2,\ZZ)\}.
\end{equation}
We make one step further in order to forget about matrices:
we will show that $L_n$ is equal to the set
\begin{equation}\label{def:Lncurly}
\Lcal_n = \{(x,y)\in\RR^2:\;x^2+y^2=n^2-4\lambda^4;\; y\equiv n\!\!\pmod{2\lambda^2},\; x\equiv 0\!\!\pmod{\lambda}\}.
\end{equation}
To show the equality, we will use the algebraic integers $u+rz_q$ and $t+sz_q$
as an intermediate step to pass from $L_n$ to $\Lcal_n$.
In particular, we will obtain that elements in either set
correspond to algebraic integers. For $L_n$, this
is straightforward from the following lemma.

\begin{lemma}\label{lemma:S2.4}
Let $\gamma\in\Gamma$,
let $x_\gamma+i\gamma$ be as in \eqref{def:xy}
and $r,u,s,t$ be as in \eqref{def:rust}. Then
\[
y_\gamma+ix_\gamma = (u+rz_q)(t+s\zqbar).
\]
In particular,
\[
|L_n| =\frac{2|\Ccal_{z_q,n}|}{|\Ocal_K^\times|} = \frac{2c_n}{|\Ocal_K^\times|} r_K(n+2\lambda^2)r_K(n-2\lambda^2),
\]
with $c_n$ as in Proposition \ref{intro:prop2}.
\end{lemma}
\begin{proof}
First expand the product on the right, which gives
\begin{equation}\label{2701:eq010}
y_\gamma+ix_\gamma = rs|z_q|^2 + ut + \mu\left(rt+us\right) - \lambda i(us - rt).
\end{equation}
Now \eqref{def:xy} and \eqref{2501:eq002} give
\[
y_\gamma = n - \frac{8\lambda^4}{q^2}\Bigl(N(u+rz_q)+N(t+sz_q) - 2\left(rs|z_q|^2+ut + \mu rt +\mu us \right)\Bigr).
\]
Since $4\lambda^2=q$, $N(u+rz_q)=n+2\lambda^2$ and $N(t+sz_q)=n-2\lambda^2$,
the right-hand side simplifies, giving the real part in \eqref{2701:eq010}.
As for $x_\gamma$, we have
\[
\begin{split}
\frac{q^2x_\gamma}{2\lambda}
={}&
4\lambda^2\left(
\mu N(u+rz_q)+\mu N(t+sz_q)
-2\mu rs|z_q|^2 - 2\mu ut + 2rt(\lambda^2-\mu^2) - 2us|z_q|^2
\right)
\\
&-4\lambda^2\mu
\left(N(u+rz_q)+N(t+sz_q) - 2\left(rs|z_q|^2+ut + \mu rt + \mu us \right)\right).
\end{split}
\]
Simplifying and using $|z_q|^2=\mu^2+\lambda^2$ and again $4\lambda^2=q$, we obtain
the imaginary part in \eqref{2701:eq010}.
\end{proof}

Elements in $\Lcal_n$ correspond to algebraic integers too.

\begin{lemma}\label{lemma:S2.5}
Let $(x,y)$ be a point in $\Lcal_n$. Then $y+ix$ is an element of $\Ocal_K$.
\end{lemma}
\begin{proof}
When $q=4,8$ this is immediate,
since $y$ is then an integer and $x$ is an integer multiple of $z_q$.
Assume $q$ is odd and write $x=\lambda h$ and $y=n+2\lambda^2k$
for some $h,k\in\ZZ$. Then
\[
n^2-4\lambda^4 = x^2 + y^2 = \lambda^2h^2 + n^2 + 4\lambda^2k + 4\lambda^4.
\]
Recalling $4\lambda^2=q$, this implies
\[
h^2 + qk^2 + 4kn + q = 0.
\]
Reducing modulo two, we deduce that $h$ and $k$ have
different parity. When $h$ is even and $k$ is odd, it follows that $y$ is an integer
and $x$ is an integer multiple of $\sqrt{q}$, say $x=x_0\sqrt{q}$, and therefore
$y+ix=y-x_0+2x_0z_q\in\Ocal_K$. When $h$ is odd and $k$ is even, then
$y$ is half an odd integer and $x$ is an odd multiple of $\lambda$, so in this case
$y+ix=y-x+xz_q\in\Ocal_K$.
\end{proof}

We can now prove the equality between the sets $L_n$ and $\Lcal_n$.

\begin{proposition}\label{prop:S2.6}
Let $z_q$ be as in \eqref{def:zq}, $n\in\Ncal_{z_q}$
and $L_n$, $\Lcal_n$ be the sets defined in \eqref{def:Ln} and \eqref{def:Lncurly}.
Then $L_n=\Lcal_n$.
In particular, if $c_n$ denotes the constant appearing in Proposition~\ref{intro:prop2},
we have
\[
|\Lcal_n| = 2c_n r_K(n^2-4\lambda^4).
\]
\end{proposition}

\begin{proof}
Regarding the inclusion $L_n\subseteq \Lcal_n$, first of all we have
$x_\gamma^2+y_\gamma^2=n^2-4\lambda^4$ by Lemma~\ref{lemma:S2.2}.
Moreover, by the definition of $x_\gamma,y_\gamma$ we see that
\[
y_\gamma = n-2\lambda^2(|z_q|^2c^2+2\mu cd+d^2) \equiv n \pmod{2\lambda^2},
\]
since the quantity in parenthesis is an integer
(it is the norm in $\Ocal_K$ of $d+cz_q$). Similarly, we have
\[
\frac{x_\gamma}{\lambda}
=
2ac|z_q|^2 + 2bd + 2\mu(ad+bc) -2\mu(|z_q|^2c^2+2\mu cd+d^2),
\]
which is an integer since $2\mu$ is either zero or one and $a,b,c,d\in\ZZ$.
Therefore $L_n\subseteq \Lcal_n$.

Now for the reverse inclusion, let $(x,y)$ be a point in $\Lcal_n$.
By Lemma \ref{lemma:S2.5} we know that $y+ix\in\Ocal_K$.
It follows in particular that
$n^2-4\lambda^4$ is a norm in $\Ocal_K$.
From standard algebraic number theory,
this is equivalent with saying that, in the factorisation of $n^2-4\lambda^4$,
all the primes $p$ such that $\chi_{q}(p)=-1$ (with $\chi_q(\cdot)=\leg{-q}{\cdot}$)
appear with even exponent.
Note that $\gcd(n+2\lambda^2,n-2\lambda^2)$ divides $q$.
Therefore, the property of the factorisation
and of being a norm descends to $n\pm 2\lambda^2$. In other words, we can find algebraic integers $z_1,z_2\in\Ocal_K$ such that
\begin{equation}\label{2701:eq001}
y+ix = z_1\overline{z_2},\qquad N(z_1)=n+2\lambda^2,\;N(z_2)=n-2\lambda^2.
\end{equation}
Write $z_1=u+rz_q$ and $z_2=t+sz_q$, for some integers $r,u,s,t$.
We claim that the pair $(u+rz_q,t+sz_q)$ belongs to the
set $\Ccal_{z_q,n}$ defined in \eqref{def:Czqn},
to prove which we need to show that \eqref{2501:eq001} holds.

We distinguish on whether $q$ is even or odd.
The case $q=4$ is treated in \cite{chatzakos_distribution_2020},
so we do not discuss it here.

\emph{Case $q=8$.}
In this case \eqref{2501:eq001} reduces (see \eqref{3103:eq001}) to the two congruences
\begin{equation}\label{0804:eq010}
r\equiv s\!\!\!\pmod{2} \qquad u\equiv t\!\!\!\pmod{4}.
\end{equation}
If $n$ is even, Proposition \ref{intro:prop2} (see also its proof) implies that the condition
\eqref{0804:eq010} is automatically satified once we have the condition on the norms,
so there is nothing to prove.
Assume $n\equiv 1\!\!\pmod{4}$. Then from the norms we get
\[
u^2+2r^2\equiv 1\pmod{4},\quad t^2+2s^2\equiv 1\pmod{4},
\]
and so $r\equiv s\equiv 0\pmod{2}$ and $u\equiv t\equiv 1\pmod{2}$.
In order to show that in fact $u\equiv t\!\!\pmod{4}$ we note that
by \eqref{2701:eq001} we also have
\[
2rs+ut = y \equiv n\equiv 1\pmod{4}.
\]
Since $2rs\equiv0\!\!\pmod{4}$, this leads to $ut\equiv 1\!\!\pmod{4}$,
therefore $u\equiv t\!\!\pmod{4}$ and \eqref{0804:eq010} holds.
Assume now $n\equiv 3\!\!\pmod{4}$.
From the conditions on the norms we get
\[
u^2+2r^2\equiv 3\pmod{4},\quad 2s^2+t^2\equiv 3\pmod{4},
\]
and so $r,u,s,t$ are all odd.
To show $u\equiv t\!\!\pmod{4}$ we argue as before
and observe that
\[
2rs+ut = y \equiv n \equiv 3\pmod{4}.
\]
Since $2rs\equiv 2\!\!\pmod{4}$, it follows again $ut\equiv 1\!\!\pmod{4}$
and thus $u\equiv t\!\!\pmod{4}$, so \eqref{0804:eq010} holds in this case too.

\emph{Case $q$ odd.}
In this case \eqref{2501:eq001} simplifies (see \eqref{3103:eq004}) to
\begin{equation}\label{0904:eq001}
r+2u \equiv s+2t \pmod{q}.
\end{equation}
The condition on the norms tells us that
\[
(r+2u)^2 \equiv 4n\equiv (s+2t)^2 \pmod{q},
\]
so $r+2u\equiv \pm s+2t$ modulo $q$. If $q|2n$ then both sides are zero mod $q$
and so \eqref{0904:eq001} automatically holds.
If $q\nmid 2n$, assume $r+2u\equiv -(s+2t)$ mod $q$.
By looking at the real part in \eqref{2701:eq001}
and recalling that $y\equiv n$ modulo $2\lambda^2$, we have (cf.~\eqref{2701:eq010})
\[
4n \equiv rs + 4ut + 2\left(rt+us\right) \equiv (r+2u)(s+2t) \equiv -(r+2u)^2 \equiv -4n \pmod{q},
\]
a contradiction. Hence \eqref{0904:eq001} holds.

Finally, the identity for $|\Lcal_n|$ follows from Lemma \ref{lemma:S2.4}
and the fact that $\frac{r_k}{|\Ocal_K^\times|}$ is multiplicative (see \S\ref{S4.1}). 
\end{proof}

\begin{remark}\label{S2:rmk}
Proposition \ref{prop:S2.6} may fail in general.
For instance, if we consider $z_q=2i$ and $n=10$,
we have $L_n=\emptyset$ whereas $\Lcal_n=\{(0,-6)\}$,
so the equality $L_n=\Lcal_n$ is not true.
\end{remark}

By combining Lemma \ref{lemma:S2.3} and Proposition \ref{prop:S2.6},
we deduce we have a $\frac{1}{2}|\Ocal_K^\times|$-to-one
conformal map from $\{\gamma z_q:\gamma\in\Gamma_{z_q,n}\}$ to $\Lcal_n$.
In particular, the angles $\{\theta(\gamma):\gamma\in\Gamma_{z_q,n}\}$
can be equivalently described
as angles of $x+iy$, for $(x,y)\in\Lcal_n$
(each one repeated $\frac{1}{2}|\Ocal_K^\times|$ times),
which proves Proposition~\ref{intro:prop1}.

We conclude this section by rewriting the discrepancy
appearing in Theorem \ref{intro:thm1}
in terms of the angles $\theta(x+iy)$
rather than $\theta(\gamma)$.
Such a new formulation will be the starting point in the
proof of Theorem \ref{intro:thm1} in Section \ref{S4}.

\begin{corollary}\label{cor:S2.7}
Let $n\in\Ncal_{z_q}$ and $\Lcal_n$ as in \eqref{def:Lncurly}.
Then
\[
\sup_{I\subseteq S^1}\left|
\frac{1}{|\Gamma_{z_q,n}|} \sum_{\gamma\in\Gamma_{z_q,n}} \mathbf{1}_{\{\theta(\gamma)\in I\}} - \frac{|I|}{2\pi}
\right|
=
\sup_{I\subseteq S^1}\left|
\frac{1}{|\Lcal_n|} \sum_{(x,y)\in\Lcal_n} \mathbf{1}_{\{\theta(x+iy)\in I\}} - \frac{|I|}{2\pi}
\right|.
\]
\end{corollary}

\section{Shifted $B$-numbers: Proof of Theorem \ref{intro:thm2}}\label{S3}

In this section we prove Theorem \ref{intro:thm2},
which is done by sieve theory methods, following the line of argument
adopted in \cite{indlekofer_scharfe_1974}.

The function $b_K(n)$ indicates whether or not $n$ is the norm of an element
in $\Ocal_K$. This is characterised starting from the behaviour on primes.
For fixed $K$, with discriminant $-q$,
we define two sets $\mathscr{D}_{1}$ and $\mathscr{D}_{-1}$ by
\[
\mathscr{D}_{\pm 1} := \{ n\in\NN:\; p|n \implies \chi_q(p)=\pm 1 \}.
\]
Here and in the rest of the paper
$\chi_q$ will denote the real character $\chi_q(\cdot)=\leg{-q}{\cdot}$.
An integer $n$ is then a norm if and only if the primes from $\mathscr{D}_{-1}$
in the factorisation of $n$ appear with even exponent.

The first step of the proof
consists in estimating from below the shifted convolution sum by a similar one
where we restrict to a suitable arithmetic progression and consider
only terms from the set $\mathscr{D}_{1}$.
More precisely, we start with the sum
\begin{equation}\label{def:Bxh}
B(x,h) := \sum_{n\leq x} b_K(n)b_K(n+h).
\end{equation}
Note that we can assume that $h$ is coprime with $q$. Indeed,
if this is not the case and $h=q^\alpha h'$ with $(h',q)=1$
(or $h=2^\alpha h'$ with $h'$ odd in the case $q=4,8$),
then by restricting to $n=q^\alpha n'$
and observing that $b_K(q^\alpha m)=b_K(m)$ for any $m\in\NN$ gives
\[
B(x,h) \geq \sum_{n'\leq x'} b_K(n')b_K(n'+h'),
\]
where $x'=xq^{-\alpha}$, which shows it suffices to lower bound
the sum on the right, where the shift $h'$ is now coprime to $q$.
Replacing $q^\alpha$ by $2^\alpha$
leads to a similar inequality in the case $q=4,8$.

Next, we assume for a moment that $q$ is odd and restrict the
summation in \eqref{def:Bxh} to a suitable arithmetic progression.
Up to replacing $h$ by $-h$ and interchanging the roles of $n$ and $n+h$,
we can assume that $h$ is a quadratic residue modulo $q$.
Once we have this, we choose $n\equiv q\!\!\pmod{q^2h}$.
If $h$ is odd, we also impose $n\equiv 4\!\!\pmod{8}$.
These choices imply $(n,n+h)=(q,h)=1$ and in particular
\[
b_K(n)b_K(n+h) = b_K(4^{-\sigma}q^{-1}n(n+h)),
\]
where $\sigma=0,1$ according to whether $h$ is even or odd, respectively,
so that the argument on the right is always odd.
Similarly, when $q=4,8$, we can assume that $h\equiv 1\!\!\pmod{4}$
or $h\equiv 1,3\!\!\pmod{8}$. Then we take $n\equiv 4q\!\!\pmod{4q^2h}$
so we can write again $b_K(n)b_K(n+h)=b_K(4^{-\sigma}q^{-1}n(n+h))$.

Write $n=n_1j+n_0$ for some integers $n_1,n_0\in\NN$
and $j\leq y:=n_1^{-1}(x-n_0)$. By further restricting to elements
in $\mathscr{D}_{1}$, we deduce that
\[
B(x,h)
\geq
\sum_{j\leq y} b_K(4^{-\sigma}q^{-1}(n_1j+n_0)(n_1j+n_0+h))
\geq
B^*(y,h),
\]
where
\begin{equation}\label{def:Bstar}
B^*(y,h) := |\{j\leq y:\;4^{-\sigma}q^{-1}(n_1j+n_0)(n_1j+n_0+h)\in\mathscr{D}_{1}\}|.
\end{equation}

Since $y\asymp x$, Theorem \ref{intro:thm2} will follow
from a lower bound for $B^*(y,h)$ of the right order of magnitude,
as stated in the proposition below.

\begin{proposition}\label{S3:prop}
Let $B^*(y,h)$ be the cardinality defined in \eqref{def:Bstar}. Then
\[
B^*(y,h) \gg_{K,h} \frac{y}{\log y}.
\]
\end{proposition}

Let $z>2$ and consider the product of primes
\[
P_q(z) := \prod_{\substack{p<z\\p\in\mathscr{D}_{-1}}} p.
\]
Define also the set
\begin{equation}\label{def:M}
M := \{4^{-\sigma}q^{-1}(n_1j+n_0)(n_1j+n_0+h):\; j\leq y\}
\end{equation}
and the function
\begin{equation}\label{def:AMz}
A(M,z) := |\{m\in M:\; (m,P_q(z))=1\}|.
\end{equation}
The quantity $A(M,z)$ is of a shape suitable to be attacked with sieve theory
and should be viewed as an approximation to $B^*(y,h)$.
When $z$ is large enough (say $z\gg y^{1/2}$),
we expect $B^*(y,h)\asymp A(M,z)$, so that Proposition \ref{S3:prop}
would follow if we could produce a lower bound for $A(M,z)$.
To make this rigorous, we will compare $B^*(y,h)$ and $A(M,z)$
for $z$ chosen as a power of $y$ slightly smaller than $1/2$ and will control
explicitly what happens with the bad primes $p\in\mathscr{D}_{-1}$
in the range $z\leq p\ll y^{1/2}$.

Before diving into the proof of Proposition \ref{S3:prop}, we discuss
three auxiliary results. The first one (Lemma \ref{S3:lemma1})
is the linear sieve, which we present in a simplified form,
specialized to our set $M$ above.
The second result (Lemma \ref{HB-lemma}) is an upper bound
for a sieve of polynomials at primes, which can be derived from Selberg's sieve.
We use such an upper bound to prove our third result, Lemma \ref{S3:lemma3}.

\begin{lemma}\label{S3:lemma1}
Let $y\geq z$ be positive real numbers and
let $s=\frac{\log y}{\log z}$. Assume that $2\leq s\leq 4$.
Let $M$ and $A(M,z)$ be as in \eqref{def:M} and \eqref{def:AMz}.
Then
\[
A(M,z)\geq y \prod_{\substack{2<p<z\\p\in\mathscr{D}_{-1}}} \left(1-\frac{2}{p}\right)\left\{\frac{2e^\gamma}{s}\log(s-1)+O\left(\frac{1}{\log y}\right)\right\},
\]
where $\gamma$ is Euler's constant.
\end{lemma}

\begin{proof}
The result is a consequence of \cite[Theorem 12.14]{friedlander_opera_2010},
let us explain briefly why (for comparison, see also
\cite[Theorem 8.4]{halberstam_sieve_1974} and \cite[Lemma 1]{indlekofer_scharfe_1974},
where the term $(\log y)^{-1}$ appears with a worse exponent).
For squarefree odd integers $d$ coprime with $qh$, the equation
\begin{equation}\label{2503:eq006}
4^{-\sigma}q^{-1}(n_1j+n_0)(n_1j+n_0+h)\equiv 0 \pmod{d}
\end{equation}
has $2^{\omega(d)}$ solutions, where $\omega(d)$ is the number of prime factors of $d$.
Thus, the congruence sums satisfy
\begin{equation}\label{2503:eq001}
A_d = |\{m\in M: m\equiv 0\pmod{d}\}| = \frac{y}{d}2^{\omega(d)} +|\theta| 2^\omega(d),\qquad |\theta|\leq 1.
\end{equation}
On the other hand, if $d$ is even or if there is a prime in $\mathscr{D}_{-1}$
dividing both $d$ and $h$, there are no solutions
(since the quantity on the left in \eqref{2503:eq006}
is always odd and $h$ divides $n_1$ by construction).
Our sifting range $P_q(z)$ contains only primes from $\mathscr{D}_{-1}$,
so the density function is defined on primes by $g(2)=0$ and, for $p\neq 2$,
\[
g(p)=
\begin{cases}
2/p & p\in\mathscr{D}_{-1},\;p\nmid h,\\
0 & \text{otherwise.}
\end{cases}
\]
We see by \eqref{2503:eq001} that \cite[(12.74)]{friedlander_opera_2010}
is satisfied by the function $g$, while by the prime number theorem in arithmetic progressions
we also have that $g$ satisfies \cite[(12.48)]{friedlander_opera_2010}.
Therefore, the hypotheses of \cite[Theorem 12.14]{friedlander_opera_2010} are satisfied
and we obtain
\[
A(M,z) \geq y\prod_{\substack{2<p<z\\p\in\mathscr{D}_{-1}\\p\nmid h}} \left(1-\frac{2}{p}\right)
\left\{f(s) + O\left(\frac{H(s)}{\log z}\right)\right\}.
\]
In particular, the inequality is valid when $s\in[2,4]$;
$f(s)$ is equal to $\frac{2e^\gamma}{s}\log(s-1)$ in this interval (see e.g.~\cite[(12.2)]{friedlander_opera_2010});
$H(s)$ is uniformly bounded for $s\in[2,4]$ (as it is a continuous function \cite[(12.39)]{friedlander_opera_2010}
on a compact set).
Upon extending the product to the odd primes dividing $h$
and noting that $\log z\gg \log y$, we obtain the lemma.
\end{proof}

\begin{lemma}[{\cite[Theorem 4.2]{halberstam_sieve_1974}}]\label{HB-lemma}
Let $F$ be a polynomial of degree $g\geq 1$ with integer coefficients.
For each prime $p$, denote by $\rho(p)$ the number of solutions of $F(n)\equiv 0\!\!\pmod{p}$.
Then, for any set of primes $\Pcal$, we have
\[
|\{p:\; p\leq x,\; (F(p),\Pcal)=1\}|
\ll
\prod_{\substack{p<x\\p\in\Pcal}} \left(1-\frac{\rho(p)}{p}\right)
\prod_{\substack{p<x\\p\in\Pcal\\p|F(0)}} \left(1-\frac{1}{p}\right)^{-1} \frac{x}{\log x},
\]
where the implied constant depends only on $g$.
\end{lemma}

Using the above lemma we can prove the following result which is tailored to our later applications.

\begin{lemma}\label{S3:lemma3}
Let $h\in\ZZ$, $h\neq 0$, and
$a=4^\sigma mp_0$, with $m\in\mathscr{D}_{1}$, $p_0\in\mathscr{D}_{-1}$,
$p_0$ prime and $\sigma\in\{0,1\}$. Set
\[
M^* := \{n=ap+h:\;p\leq x/a,\;p \text{ prime}\}.
\]
Then for $2\leq \sqrt{x/a}\leq z\leq x/a$ we have
\[
|\{n\in M^*:\; (n,P_q(z))=1\}|
\ll_h \frac{x}{a}\left(\log\frac{x}{a}\right)^{-3/2}.
\]
The implied constant depends on $h$ but is independent of $a,x$ and $z$.
Furthermore, the result still holds when $2\in\mathscr{D}_{-1}$ and $P_q(z)$ is replaced by $2^{-1}P_q(z)$.
\end{lemma}

\begin{proof}
The quantity we want to bound is
\[
A(M^*,z):=|\{p\leq \frac{x}{a}:\;(ap+h,P_q(z))=1\}|.
\]
When $p_0<z$, we may assume that $p_0\nmid h$ for otherwise this cardinality is zero.
We apply Lemma \ref{HB-lemma} with the linear polynomial $F(n)=an+h$,
the set of primes $\Pcal=\{p\in\mathscr{D}_{-1}:\;p<z\}$ and with $x/a$ in place of $x$,
obtaining
\[
A(M^*,z)
\ll
\prod_{\substack{p<z\\p\in\mathscr{D}_{-1}}} \left(1-\frac{\rho(p)}{p}\right)
\prod_{\substack{p<z\\p\in\mathscr{D}_{-1}\\p|h}} \left(1-\frac{1}{p}\right)^{-1} \frac{x}{a\log (x/a)}.
\]
By extending the second product to all the divisors of $h$ we obtain,
with an implicit constant that depends only on $h$,
\begin{equation}\label{1903:eq001}
A(M^*,z) \ll_h
\prod_{\substack{p<z\\p\in\mathscr{D}_{-1}}} \left(1-\frac{\rho(p)}{p}\right) \frac{x}{a\log (x/a)}.
\end{equation}
As for the value of $\rho(p)$, we observe that $F(n)\equiv 0\!\!\pmod{p}$
has one solution for all $p<z$, $p\in\mathscr{D}_{-1}$, unless $p=p_0$,
in which case it has no solutions, or $p=2$, $\sigma=1$ and $h$ is even,
in which case again there are no solutions.
Therefore, by the prime number theorem in arithmetic progressions
and using that $z\geq \sqrt{x/a}$, we have
\[
\sum_{\substack{p<z\\p\in\mathscr{D}_{-1}}} \frac{\rho(p)}{p}
=
\sum_{\substack{p<z\\p\in\mathscr{D}_{-1}}} \frac{1}{p} + O(1)
\geq
\frac{1}{2}\log\log \frac{x}{a} + O(1).
\]
As a consequence, we deduce
\[
\prod_{\substack{p<z\\p\in\mathscr{D}_{-1}}} \biggl(1-\frac{\rho(p)}{p}\biggr)
\ll
\exp\biggl(-\frac{1}{2}\log\log \frac{x}{a}\biggr)\ll \biggl(\log \frac{x}{a}\biggr)^{-1/2}.
\]
Inserting this in \eqref{1903:eq001} gives the desired estimate.
The final part of the lemma follows by the same proof with minor modifications.
\end{proof}

\begin{proof}[Proof of Proposition \ref{S3:prop}]
In order to show that $B^*(y,h)$ is large, we wish to show
that there are many elements of $\mathscr{D}_{1}$
in the set $M$ defined in \eqref{def:M}. To approach this, we consider
sets that approximate $\mathscr{D}_{1}$: for $z>2$ and any $l\geq 1$, define
\[
\begin{gathered}
\mathscr{D}_{1}(z) := \{m\in\NN:\; p|m,\;p\text{ prime}, p<z \Rightarrow p\in\mathscr{D}_{1}\},\\
\mathscr{D}_{1l}(z) := \{n=m\prod_{i=1}^{l}p_i:\; m\in\mathscr{D}_{1},\;p_i>z, \;p_i\text{ prime}, p_i\in\mathscr{D}_{-1}\;\forall\; i=1,\dots,l\}.
\end{gathered}
\]
We then have the inclusions, for every $l\geq 1$ and $z_1<z_2$,
\begin{equation}\label{2503:eq002}
\mathscr{D}_{1l}(z) \subseteq \mathscr{D}_{1}(z),
\quad
\mathscr{D}_{1}\subseteq \mathscr{D}_{1}(z_1) \subseteq \mathscr{D}_{1}(z_2).
\end{equation}
Pick $z=y^{1/s}$ with $2<s<5/2$. Let $A(M,z)$ be the quantity defined in \eqref{def:AMz}.
By looking at the size of possible harmful primes, we see that
\begin{equation}\label{2503:eq003}
A(M,z) = B^*(y,h)
\;+ \sum_{\substack{m\in M:\\m\in\mathscr{D}_{12}(y^{1/s})}} \!\!\!\!1
\;\;+ \sum_{\substack{m\in M:\\m\in\mathscr{D}_{14}(y^{1/s})}} \!\!\!\!1.
\end{equation}
Now, $m$ is the product of two integers, $4^{-\sigma}q^{-1}(n_1j+n_0)=:m_1$ and $n_1j+n_0+h=:m_2$, say.
Because of our choice of arithmetic progression,
we have $\chi_q(m_1)=\chi_q(m_2)=1$.
Therefore, primes from $\mathscr{D}_{-1}$ will divide each of $m_1$ and $m_2$ in pairs.
In turn, this means that the sums above decompose further as
\[
\sum_{\substack{j\leq y:\\4^{-\sigma}q^{-1}(n_1j+n_0)\in\mathscr{D}_{12}(y^{1/s})\\n_1j+n_0+h\in\mathscr{D}_{1}}} \!\!\!\! 1
\quad+
\sum_{\substack{j\leq y:\\4^{-\sigma}q^{-1}(n_1j+n_0)\in\mathscr{D}_{1}\\n_1j+n_0+h\in\mathscr{D}_{12}(y^{1/s})}} \!\!\!\!\!\! 1
\quad+
\sum_{\substack{j\leq y:\\4^{-\sigma}q^{-1}(n_1j+n_0)\in\mathscr{D}_{12}(y^{1/s})\\n_1j+n_0+h\in\mathscr{D}_{12}(y^{1/s})}} \!\!\!\!\! 1.
\]
Denote the three sums by $\Scal_1,\Scal_2,\Scal_3$, respectively.
Treating first $\Scal_1$ and $\Scal_3$, we can bound, recalling the inclusions in \eqref{2503:eq002}
\[
\Scal_1+\Scal_3\ll\!\!\!\!\!\!\!\!
\sum_{\substack{j\leq y:\\4^{-\sigma}q^{-1}(n_1j+n_0)\in\mathscr{D}_{12}(y^{1/s})\\n_1j+n_0+h\in\mathscr{D}_{1}(y^{1/s})}} \!\!\!\!\!\!\!\!\!\!\!\! 1
\quad\ll
\sum_{\substack{m\leq n_1y^{1-2/s}+n_0\\m\in\mathscr{D}_{1}}} \;\;
\sum_{\substack{y^{1/s}<r\leq\frac{n_1y^{1-1/s}+n_0}{m}\\r\;prime,\, r\in\mathscr{D}_{-1}}} \;\;
\sum_{\substack{y^{1/s}<p\leq\frac{n_1y+n_0}{mr}\\4^\sigma mrp+h\in\mathscr{D}_{1}(y^{1/s})\\p\;prime}} 1.
\]
By applying Lemma \ref{S3:lemma3} to the last sum with $x=n_1y+n_0$, $a=4^\sigma mr$ and $z=y^{1/s}$
(so that for $2<s<5/2$ and $y$ sufficiently large we have $\sqrt{x/a}\leq z\leq x/a$),
we can bound the above by
\[
\ll
\frac{y}{(\log y)^{3/2}}
\sum_{\substack{m<n_1y^{1-2/s}+n_0\\m\in\mathscr{D}_1}}\;\;
\sum_{\substack{y^{1/s}<r<\frac{n_1y^{1-1/s}+n_0}{m}\\r\;prime,\,r\in\mathscr{D}_{-1}}} \frac{1}{mr}.
\]
The sum over $r$ can be easily bounded by using the prime number theorem and gives
\[
\ll
\frac{y}{(\log y)^{3/2}}(\log(s-1)+o(1))
\sum_{\substack{m<n_1y^{1-2/s}+n_0\\m\in\mathscr{D}_1}} \frac{1}{m}.
\]
Finally, since
\[
\sum_{\substack{m\leq x\\m\in\mathscr{D}_{1}}} \frac{1}{m}
\leq
\prod_{\substack{p\leq x\\p\in\mathscr{D}_{1}}} \left(1-\frac{1}{p}\right)^{-1} \ll (\log x)^{1/2},
\]
we arrive at the estimate
\begin{equation}\label{2503:eq004}
\Scal_1 + \Scal_3 \ll \frac{y}{(\log y)}(\sqrt{s-2}\log(s-1)+o(1)).
\end{equation}
The sum $\Scal_2$ is bounded similarly after swapping the roles of $n_1j+n_0$ and
$n_1j+n_0+h$, since we have the estimate
\[
\Scal_2 \ll \sum_{\substack{n_1j+n_0+h\in\mathscr{D}_{12}(y^{1/s})\\n_1j+n_0\in 4\mathscr{D}_{1}(y^{1/s})}} 1,
\]
so that when $2\in\mathscr{D}_{-1}$ the last part of Lemma \ref{S3:lemma3}, with the negative shift $-h$, is to be applied.
Going back to \eqref{2503:eq003}, we can apply Lemma \ref{S3:lemma1} to obtain the lower bound
\begin{equation}\label{2503:eq005}
A(M,y^{1/s}) \gg \frac{y}{\log y}(\log(s-1)+o(1)).
\end{equation}
Combining \eqref{2503:eq003} and \eqref{2503:eq004}--\eqref{2503:eq005}, we arrive at the inequality
\[
B^*(y,h) \geq \frac{y}{\log y}(c_1 - c_2\sqrt{s-2} +o(1)) \log (s-1),
\]
where $c_1,c_2>0$ depend on $K$ and $h$.
Picking $s$ sufficiently close to $2$ gives a positive constant on the right
and proves the proposition.
\end{proof}

\section{Proof of Theorem \ref{intro:thm1}}\label{S4}

In this section we prove Theorem \ref{intro:thm1}.
In \S\ref{S4.1} we collect a few preliminary results
and introduce a function $v_k$ defined as an exponential sum
over elements in $\Ocal_K$, proving it is multiplicative.
Then in \S\ref{S4.2} we give the proof of Theorem \ref{intro:thm1},
using the results of \S\ref{S4.1} as well as those
of Sections~\ref{S2} and \ref{S3}.

\subsection{Preliminary results}\label{S4.1}

Let $M\geq 1$ and let $r_K(M)$ denote the number of algebraic integers in $\Ocal_K$
with norm $M$. By \cite[(11.9)]{iwaniec_topics_1997} we have
\begin{equation}\label{2604:eq001}
\frac{r_K(M)}{|\Ocal_K^\times|} = \sum_{d|M} \chi_q(d),
\end{equation}
where $|\Ocal_K^\times|$ is the number of units in $\Ocal_K$.
In particular, the above is a multiplicative function, so if we define
\begin{equation}\label{def:omegaK}
\omega_K(M) = \sum_{\substack{p|M\\\chi_q(p)=1}} 1
\quad\text{and}\quad
\Omega_K(M) = \sum_{\substack{p^a||M\\\chi_q(p)=1}} a,
\end{equation}
then we have
\begin{equation}\label{2104:eq005}
2^{\omega_K(M)} \leq \frac{r_K(M)}{|\Ocal_K^\times|} \leq 2^{\Omega_K(M)}.
\end{equation}

Assume now that $M$ is a norm. When $q\neq 8$, this implies
that $M$ is a quadratic residue modulo $q$ and so we can write $M\equiv m^2\pmod{q}$
for some integer $m$.
When $q=8$, we set $m=1$ or $m=0,2$ according to whether $M$ is odd
or $M\equiv 0,2\pmod{8}$, $M\equiv 4,6\pmod{8}$ respectively.
With this notation, we define the function
\begin{equation}\label{def:rKstar}
r_K^\star(M) := \sum_{\substack{y+ix\in\Ocal_K\\N(y+ix)=M\\2y\equiv 2m\!\!\!\!\pmod{q}}} 1.
\end{equation}

\begin{lemma}\label{lemma:S4.1}
Let $M\geq 1$ be a norm in $\Ocal_K$.
The function $r_K^\star$ given in \eqref{def:rKstar}
is well defined and we have
\[
r_K^\star(M) =
\begin{cases}
r_K(M) & \text{if } (M,q)>1,\\
\frac{1}{2}r_K(M) & \rule{0pt}{12pt}\text{if } (M,q)=1.
\end{cases}
\]
\end{lemma}

\begin{proof}
Assume first that $q$ is odd and write $M\equiv m^2\pmod{q}$.
If $y+ix\in\Ocal_K$, then $2x$ is an integer multiple of $\sqrt{q}$,
so if $N(y+ix)=M$ we deduce
\[
4m^2 \equiv 4M = 4y^2+4x^2\equiv 4y^2\pmod{q}.
\]
Therefore, $y\equiv \pm m\pmod{q}$.
If $q|M$, then $m\equiv 0\pmod{q}$ and $y\equiv m\pmod{q}$ is satisfied
by all points $y+ix\in\Ocal_K$ of norm $M$, so $r_K^\star(M)=r_K(M)$ in this case.
If $q\nmid M$, then $m\not\equiv 0\pmod{q}$
and for exactly half of the units $u\in\Ocal_K^\times$ the point $u(y+ix)$
will contribute to the sum defining $r_K^\star(M)$, so $r_K^\star(M)=\frac{1}{2}r_K(M)$.
Replacing $m$ by $-m$ amounts to taking the other units,
so the value of $r_K^\star(M)$ is unchanged and the function is well defined.

For the case $q=4$ we refer to \cite[(3.2)]{chatzakos_distribution_2020}, so assume $q=8$.
If $M$ is even then by reducing modulo $8$ the identity $y^2+2x^2=M$ we deduce that $y\equiv m\pmod{4}$
always, so $r_K^\star(M)=r_K(M)$ in this case.
If $M$ is odd, then the condition on the norm only gives $y\equiv \pm 1\pmod{4}$ and therefore again half of the
elements in $\Ocal_K$ with norm $M$ satisfy the congruence in \eqref{def:rKstar},
so $r_K^\star(M)=\frac{1}{2}r_K(M)$.
\end{proof}

Next, for $X\geq 1$ we define an index set
\begin{equation}\label{def:NqX}
\IndexSet_{q}(X) :=
\begin{cases}
\{n\in\NN, n\leq X\} & \text{if $q$ is even},\\
\{n=\tfrac{2k+1}{2} \leq X,\; k\in \NN\} & \text{if $q$ is odd}.
\end{cases}
\end{equation}
In this notation, the set $\Ncal_{z_q}$ of non-empty arithmetic radii defined in \eqref{intro:def:Nzq}
is a subset of $\IndexSet_{q}(\infty)$ by Lemma \ref{lemma:S2.4} and Proposition \ref{prop:S2.6};
up to hight $X$, we have
\begin{equation}\label{1406:eq001}
\Ncal_{z_q}(X) = \{n\in\Ncal_{z_q}:n\leq X\} = \{n\in\IndexSet_{q}(X):\;b_K(n^2-4\lambda^4)=1\}.
\end{equation}

In Lemma \ref{lemma:S4.2} below we show that, for the generic $n\in\IndexSet_{q}(X)$
such that $b_K(n^2-4\lambda^4)=1$, both
$\omega_K(n^2-4\lambda^4)$ and $\Omega_K(n^2-4\lambda^4)$ are close to $\log\log X$.
The lemma is proved with the aid of a result by Nair and Tenenbaum on 
shifted convolution sums of multiplicative functions \cite[(7)]{nair_short_1998}
(see also \cite[Theorem 15.6]{friedlander_opera_2010}),
which states that for any non-negative multiplicative functions $f$ and $g$
such that $f(n),g(n)\leq \tau_l(n)$ for some divisor function $\tau_l$,
and for any integer $h$, we have
\begin{equation}\label{NT}
\sum_{n\leq X} f(n)g(n+h) \ll_h X \prod_{p\leq X} \left(1+\frac{f(p)-1}{p}\right)\left(1+\frac{g(p)-1}{p}\right).
\end{equation}

\begin{lemma}\label{lemma:S4.2}
Fix $\epsilon\in(0,1/2)$ and let $\omega_K,\Omega_K$ be as in \eqref{def:omegaK}.
Then, as $X\to\infty$, we have
\begin{equation}\label{2104:eq004}
\sum_{\substack{n\in\IndexSet_{q}(X)\\\omega_K(n^2-4\lambda^4)\leq (1-\epsilon)\log\log X}} \!\!\!\!\!\!\!\!\!\!\!\!\!\!\!\!\!\! b_K(n^2-4\lambda^4)
\quad+\!\!\!\!\!\!\!\!\!\!\!\!
\sum_{\substack{n\in\IndexSet_{q}(X)\\\Omega_K(n^2-4\lambda^4)\geq (1+\epsilon)\log\log X}} \!\!\!\!\!\!\!\!\!\!\!\!\!\!\!\!\!\! b_K(n^2-4\lambda^4)
\ll_{K,\epsilon} \frac{X}{X^{1+\frac{1}{3}\epsilon^2}}.
\end{equation}
Therefore,
\begin{equation}\label{2104:eq003}
r_K^\star(n^2-4\lambda^4) \asymp (\log n)^{\log 2 +o(1)}
\end{equation}
for $n\in\IndexSet_{q}(X)$ outside of an exceptional set of size at most $X(\log X)^{-1-\frac{1}{3}\epsilon^2}$.
\end{lemma}

\begin{proof}
Let $\alpha\in(0,1)$. By Chernoff's bound, we have
\[
\sum_{\substack{n\in\IndexSet_{q}(X)\\\omega_K(n^2-4\lambda^4)\leq (1-\epsilon)\log\log X}} \!\!\!\!\!\!\!\!\!\!\!\!\!\!\!\! b_K(n^2-4\lambda^4)
\leq
(\log X)^{\alpha(1-\epsilon)} \sum_{n\in\IndexSet_{q}(X)} b_K(n^2-4\lambda^4)e^{-\alpha\omega_K(n^2-4\lambda^4)}.
\]
Note that the function $b_K(\cdot)e^{-\alpha\omega_K(\cdot)}$ is multiplicative.
Since we can write $n^2-4\lambda^4=(n+2\lambda^2)(n-2\lambda^2)$ and the two factors
share at most a divisor of $q$, the above can be bounded by
\[
\leq
e^\alpha(\log X)^{\alpha(1-\epsilon)} \sum_{n\in\IndexSet_{q}(X)} b_K(n+2\lambda^2)e^{-\alpha\omega_K(n+2\lambda^2)}b_K(n-2\lambda^2)e^{-\alpha\omega_K(n-2\lambda^2)}.
\]
We then apply \eqref{NT} and Mertens' theorem to bound
\[
\begin{split}
&\ll_{K}
X(\log X)^{\alpha(1-\epsilon)} \prod_{p\leq X} \left(1+\frac{b_K(p)e^{-\alpha}-1}{p}\right)^2
\\
&\ll_K
\frac{X}{\log X} (\log X)^{\alpha(1-\epsilon)-1+e^{-\alpha}}.
\end{split}
\]
Picking $\alpha=\epsilon$ and using that $e^{-\epsilon}\leq 1-\epsilon+\epsilon^2/2$ we obtain the result for $\omega_K$.
The argument for $\Omega_K$ is similar, except that in the end
the exponent gives $-\epsilon(1+\epsilon)-1+e^\epsilon\leq -\epsilon^2/3$.
Finally, \eqref{2104:eq003} follows from \eqref{2104:eq004},
Lemma \ref{lemma:S4.1} and the bounds in \eqref{2104:eq005}.
\end{proof}

We conclude this part by introducing a new
function closely related to $r_K^\star$.
Let $M\geq 1$ be a norm in $\Ocal_K$ and let $m$ denote the same integer used
to define $r_K^\star(M)$, see before \eqref{def:rKstar}.
For any $k\in\ZZ$ we define
\begin{equation}\label{def:vk}
v_k(M) := \frac{1}{r_K^\star(M)} \Biggl| \sum_{\substack{y+ix\in\Ocal_K\\N(y+ix)=M\\2y\equiv 2m\!\!\!\!\pmod{q}}} e^{ik\theta(y+ix)}\Biggr|.
\end{equation}
When $r_K^\star(M)=0$, we define $v_k(M)=0$.
Note that selecting $-m$ instead of $m$ leads to a new sum which corresponds
to the original one multiplied by an element of absolute value one.
Therefore $v_k$ is well defined.
Note also that $v_k$ is real since $y+ix$ appears in the sum
if and only if $y-ix$ does. Hence, we have $v_{k}(M)=v_{-k}(M)$.

\begin{lemma}\label{lemma:S4.3}
Let $v_k$ be the function defined in \eqref{def:vk}.\\
$(i)$ $v_k$ is multiplicative.\\
$(ii)$ if $q=3$, then $v_k=0$ when $3\nmid k$.\\
$(iii)$ if $q=4$, then $v_k=0$ when $k$ is odd.\\
$(iv)$ if $q$ is odd (resp.~even), then $v_k(q^{a})=v_k(q^b)$
(resp.~$v_k(2^{a})=v_k(2^b)$), for all $a,b\geq 0$.
\end{lemma}

\begin{proof}
$(i)$ Assume $q\neq 3,4$.

Let $M_1,M_2\geq 1$ with $(M_1,M_2)=1$ and $r_K^\star(M_1),r_K^\star(M_2)\neq 0$.
For $j=1,2$, write $M_j\equiv m_j^2\pmod{q}$. Then
\begin{equation}\label{2804:eq001}
v_k(M_1)v_k(M_2) = \frac{1}{r_K^\star(M_1)r_K^\star(M_2)}
\Biggl|\sum_{\substack{N(y_1+ix_1)=M_1\\2y\equiv 2m_1\;(q)}} \sum_{\substack{N(y_2+ix_2)=M_2\\2y\equiv 2m_2\;(q)}} \!\!e^{ik\theta_1}e^{ik\theta_2}\Biggr|,
\end{equation}
where for brevity we wrote $\theta_j$ in place of $\theta(y_j+ix_j)$, $j=1,2$.
Consider the product $y+ix:=(y_1+ix_1)(y_2+ix_2)$. Since the norm is multiplicative,
$N(y+ix)=M_1M_2$. Moreover, $\theta(y+ix)\equiv \theta_1+\theta_2$ modulo $2\pi$.
Regarding the congruence modulo $q$, we distinguish two cases depending on whether
$q$ is odd or $q=8$.

When $q$ is odd, recalling that $2x_1$ and $2x_2$ are integer multiples of $\sqrt{q}$, we can write
\[
4y=4y_1y_2-4x_1x_2 \equiv 4y_1y_2 \equiv 4m_1m_2\pmod{q}.
\]
Note that $2$ is invertible modulo $q$,
so this is equivalent to $2y\equiv 2m_1m_2\pmod{q}$.
In other words, $y+ix$ satisfies the congruence in the definition of $v_k(M_1M_2)$.
Finally, since $r_K^\star$ is multiplicative when $q\neq 3,4$
(by Lemma \ref{lemma:S4.1} and \eqref{2604:eq001}),
we deduce that \eqref{2804:eq001} equals~$v_k(M_1M_2)$.

When $q=8$, the argument is the same except when $M_1\equiv M_2\equiv 3\pmod{8}$.
In this case, the above calculations lead to $y=y_1y_2-x_1x_2\equiv -1\pmod{4}$,
whereas in the definition of $v_k$ we had $y\equiv 1\pmod{4}$.
However, we saw that changing the residue class to its negative amounts
to multiplying the sum by an element of modulus 1, so
after taking absolute value we obtain again $v_k(M_1M_2)$.

$(ii)$ The sum defining $v_k$ contains the elements $y+ix$,
$(y+ix)e^{2\pi i/3}$ and $(y+ix)e^{-2\pi i/3}$. Therefore,
when $3\nmid k$, the contribution of any such three points vanishes, giving $v_k=0$.
Assume $3|k$ and $3\nmid M$ (if $3|M$ the congruence in the definition of $v_k$
is always satisfied and the argument simplifies).
Then $v_k$ can be written as a sum of \emph{primary} elements
(see e.g.~\cite[Ch.9 \S3]{ireland_classical_1990}, although they choose the residue class $2$ modulo $3$), namely we have
\[
v_k(M) = \frac{3}{r_K^\star(M)} \Biggl|\sum_{\substack{N(y+ix)=M\\y+ix\equiv 1\;(3)}} e^{ik\theta(y+ix)}\Biggr|,
\]
which can be showed to be multiplicative by a similar argument as the one
used to prove~$(i)$.

$(iii)$ When $q=4$ the sum defining $v_k$ contains both $y+ix$ and $-y-ix$,
so if $k$ is odd the corresponding exponentials cancel out and $v_k=0$.
When $k$ is even, the sum is expressed in terms of primary elements in the 
Gaussian integers, i.e.~of the form $y+ix\equiv 1\pmod{2(1+i)}$,
and is again multiplicative. We refer to
\cite[\S3.1]{chatzakos_distribution_2020} for more details.

$(iv)$ Assume $q\neq 3,4,8$. Since $q$ ramifies in $\Ocal_K$, the elements
of norm $q$ are $\pm\alpha$, where $\alpha=i\sqrt{q}\in\Ocal_K$.
Similarly, the elements of norm $q^a$ are $\pm\alpha^a$.
In particular, $r_K^\star(q^a)=2$ and
\[
2v_k(q^a) = |i^{ka}+(-1)^{k}i^{ka}| = |1+(-1)^k|,
\]
which is independent of $a$. Therefore $v_k(q^a)$ attains the same
value for all $a\geq 0$.
When $q=8$, the argument is the same except that one works with $\alpha=i\sqrt{2}$.
For the case $q=4$ we refer to \cite[\S3.2]{chatzakos_distribution_2020}.
Finally, when $q=3$ we have six units of the form $\omega^j$,
where $\omega=\frac{1+i\sqrt{3}}{2}$ and $j=0,\dots,5$.
In particular, $r_K^\star(3^a)=6$.
Recall that we can assume $3|k$, say $k=3h$, for otherwise $v_k$
is identically zero. If we take $\alpha=\frac{3+i\sqrt{3}}{2}$
(which has norm $3$ and angle $\frac{\pi}{6}$), we deduce
\[
6v_k(3^a) = \Bigl|\sum_{j=0}^{5} e^{3hi(\frac{\pi a}{6}+\frac{\pi j}{3})}\Bigr|
=\Bigl|\sum_{j=0}^{5} (-1)^{hj}\Bigr|,
\]
which again is independent of $a$.
\end{proof}

If we evaluate $v_k$ on primes, we obtain simply a cosine. More precisely, 
$v_k(p)=|\cos(k\theta_p)|$, where $\theta_p$ is the angle of
a prime element in $\Ocal_K$ above $p$ (a primary element if $q=3,4$).
Such angles equidistribute on the unit circle (as a consequence
of the prime number theorem over number fields,
see \cite[Theorem 5.36]{iwaniec_analytic_2004} and \cite{kowalski_errata}).
Following an argument of Erd\H{o}s and Hall \cite[(24)--(25)]{erdos_angular_1999},
one can derive the following asymptotic for the absolute value of the cosines.
Let $k\neq 0$ be given. Then, uniformly for $|k|\ll \log x$, we have
\begin{equation}\label{EH}
\sum_{\substack{p\leq x\\\chi_q(p)=1}} \frac{|\cos(k\theta_p)|}{p} = \frac{1}{\pi}\log\log x + O(1).
\end{equation}

\subsection{Proof of Theorem \ref{intro:thm1}}\label{S4.2}
In Corollary \ref{cor:S2.7} we saw that the discrepancy appearing in Theorem \ref{intro:thm1}
can be calculated by looking at angles of $x+iy$, where $(x,y)\in\Lcal_n$,
the set $\Lcal_n$ is defined in \eqref{def:Lncurly} and $n\in\Ncal_{z_q}$ as in \eqref{intro:def:Nzq}.
Furthermore, we note that the discrepancy is unchanged if we replace $x+iy$ by $y+ix$,
which is more convenient to have, since in Lemma~\ref{lemma:S2.5} we showed that $y+ix\in\Ocal_K$
and so we can use the results from \S\ref{S4.1}.
In other words, to prove Theorem \ref{intro:thm1} it suffices to estimate
\[
\EuScript{D}_n := \sup_{I\subseteq S^1}\left|
\frac{1}{|\Lcal_n|} \sum_{y+ix\in\Lcal_n} \mathbf{1}_{\{\theta(y+ix)\in I\}} - \frac{|I|}{2\pi}
\right|.
\]
Here and below, with a slight abuse of notation, we write $y+ix\in\Lcal_n$ rather than $(x,y)\in\Lcal_n$.

By Proposition \ref{prop:S2.6}, $|\Lcal_n|=r_K^\star(n^2-4\lambda^4)$, where $r_K^\star$
is as in \eqref{def:rKstar}. Even more precisely, the elements $y+ix\in\Lcal_n$ are exactly those appearing
in the definition of $r_K^\star$ and $v_k$ in \eqref{def:vk}, with $M=n^2-4\lambda^4$.
Therefore, applying the Erd\H{o}s--Tur\'an inequality \cite[Corollary 1.1]{montgomery_ten_1994},
we can bound
\begin{equation}\label{2804:eq011}
\EuScript{D}_n \ll \frac{1}{\log X} + \sum_{k=1}^{\log X} \frac{v_k(n^2-4\lambda^4)}{k}.
\end{equation}

Let $\epsilon\in(0,1/2)$.
We now restrict $n$ to a density one subset of $\Ncal_{z_q}$, namely to
\[
\EuScript{B}_K(X;\epsilon) := \{n\in\Ncal_{z_q}: n\leq X\text{ and }\omega_K(n^2-4\lambda^4)\geq (1-\epsilon)\log \log X\}.
\]
Using the notation $\IndexSet_{q}(X)$ from \eqref{def:NqX} and recalling \eqref{1406:eq001},
the set $\EuScript{B}_K(X;\epsilon)$ can be written as
\[
\EuScript{B}_K(X;\epsilon) = \{n\in\IndexSet_{q}(X):\; b_K(n^2-4\lambda^4)=1,\; \omega_K(n^2-4\lambda^4)\geq(1-\epsilon)\log\log X\}.
\]

Note that
\[
\#\{n\in\IndexSet_{q}(X): b_K(n^2-4\lambda^4)=1\}
\geq \sum_{n\in\IndexSet_{q}(X)} b_K(n+2\lambda^2)b_K(n-2\lambda^2)
\gg_K \frac{X}{\log X},
\]
the last inequality being true by Theorem \ref{intro:thm2}.
On the other hand, Lemma \ref{lemma:S4.2} gives
\[
\#\{n\in\IndexSet_{q}(X): b_K(n^2-4\lambda^4)=1 \text{ and }\omega_K(n^2-4\lambda^4)\leq (1-\epsilon)\log\log X\}
\ll
\frac{X}{(\log X)^{1+\frac{1}{3}\epsilon^2}},
\]
which shows that $\EuScript{B}_K(X;\epsilon)$ is a density one subset of $\{n\in\IndexSet_{q}(X):\;b_K(n^2-4\lambda^4)=1\}$.
Furthermore, we distinguish between elements in $\EuScript{B}_K(X;\epsilon)$ \enquote{coprime} to $q$
or not. Set
\[
\EuScript{B}_K^{\flat}(X;\epsilon) :=
\begin{cases}
\{n\in\EuScript{B}_K(X;\epsilon):\;(n,q)=1\} & \text{$q$ even}\\
\{n\in\EuScript{B}_K(X;\epsilon):\;(2n,q)=1\} & \text{$q$ odd},
\end{cases}
\quad
\EuScript{B}_K^{\sharp}(X;\epsilon):= \EuScript{B}_K(X;\epsilon)\setminus\EuScript{B}_K^{\flat}(X;\epsilon).
\]
Define also $\IndexSet_{q}^\flat(X)$ and $\IndexSet_{q}^\sharp(X)$ in a similar way.
We will carry out the proof first for $\EuScript{B}_K^{\flat}(X;\epsilon)$ and show
how to reduce back to this case when dealing with $\EuScript{B}_K^{\sharp}(X,\epsilon)$.

By Chebyshev's inequality, we can write
\begin{equation}\label{2804:eq010}
\begin{split}
\#\bigg\{ n\in \EuScript{B}_K^{\flat}(X;\epsilon): & \sum_{1\leq k \leq \log X}\frac{v_k(n^2-4\lambda^4)}{k}\geq (\log X)^{-\log\frac{\pi}{2}+\varepsilon} \bigg\} \\
&\leq (\log X)^{\log\frac{\pi}{2}-\varepsilon}
\sum_{1\leq k \leq \log X} \frac{1}{k}
\sum_{\substack{n\in \IndexSet_q^\flat(X)\\ \omega_K(n^2-4\lambda^4)\geq (1-\varepsilon)\log\log X}}
\!\!\!\!\!\!\!\!\!\!\!\!\!\!\!\!\!\! v_k(n^2-4\lambda^4).
\end{split}
\end{equation}
We focus on the inner sum over $n$, which we denote by $S_k$.
Let $\alpha\in(0,1)$. By Chernoff's bound, we can estimate
\[
S_k \leq (\log X)^{-\alpha(1-\epsilon)} \sum_{n\in\IndexSet_q^{\flat}(X)} v_k(n^2-4\lambda^4)e^{\alpha \omega_K(n^2-4\lambda^4)}.
\]
Moreover, $(n+2\lambda^2,n-2\lambda^2)=1$ since the gcd divides $q$ but $n\in\IndexSet_q^{\flat}(X)$.
By Lemma \eqref{lemma:S4.3}, the function $v_k$ is multiplicative. Therefore,
the right-hand side above equals
\[
(\log X)^{-\alpha(1-\epsilon)} \sum_{n\in\IndexSet_q^{\flat}(X)}
v_k(n+2\lambda^2)e^{\alpha \omega_K(n+2\lambda^2)} v_k(n-2\lambda^2)e^{\alpha \omega_K(n-2\lambda^2)}.
\]
Applying \eqref{NT}, \eqref{EH} and picking $\alpha=\log\pi/2$, we obtain
\[
\begin{split}
S_k \ll \frac{X}{(\log X)^{\alpha(1-\epsilon)}} \prod_{p\leq X} \left(1+\frac{e^\alpha v_k(p)b_K(p)-1}{p}\right)^2
\ll \frac{X}{(\log X)^{1+(1-\epsilon)\log\frac{\pi}{2}}},
\end{split}
\]
Inserting this bound for $S_k$ in \eqref{2804:eq010}
and going back to $\EuScript{D}_n$ and $r_K^\star(n^2-4\lambda^4)$
by means of \eqref{2804:eq011} and \eqref{2104:eq003}, we deduce that
\[
\EuScript{D}_n \ll (\log X)^{-\log \frac{\pi}{2}+\epsilon} \ll r_K^\star(n^2-4\lambda^4)^{-C+\epsilon},
\quad
C = \frac{\log(\pi/2)}{\log 2},
\]
for all arithmetic radii $n\in\Ncal_{z_q}\cap \IndexSet_q^\flat(X)$
outside of a set of cardinality
\[
\ll \frac{X}{(\log X)^{1+\frac{1}{3}\epsilon^2}} + \frac{X(\log\log X)}{(\log X)^{1+\epsilon-\epsilon\log\frac{\pi}{2}}},
\]
which proves Theorem \ref{intro:thm1} in the coprime case.

When $n\in\EuScript{B}_K^\sharp$, assuming $q$ is odd,
then $(n+2\lambda^2,n-2\lambda^2)=q$. Therefore, by Lemma \ref{lemma:S4.3} $(iv)$,
we can write
\begin{equation}\label{0905:eq001}
v_k(n^2-4\lambda^4)=v_k\left(\frac{n+2\lambda^2}{q}\right)v_k\left(\frac{n-2\lambda^2}{q}\right),
\end{equation}
where the two arguments are now coprime and one can apply the same proof as above.
The case $q$ even is analogous, except that in \eqref{0905:eq001}
one divides by an appropriate power of $2$.



\begin{thebibliography}{00}

\bibitem{bernays_uber_1912}
P.~Bernays, \emph{\"Uber die Darstellung von positiven, ganzen Zahlen durch die primitiven, bin\"aren quadratischen Formen einer nicht quadratischen Diskriminante},
PhD Thesis, G\"ottingen Universit\"ats-Buehdruckerei, 1912.

\bibitem{boca}
F.~P.~Boca, \emph{Distribution of angles between geodesic rays associated with hyperbolic lattice points},
Q.~J.~Math.~\textbf{58} (2007), no.~3, 281--295.

\bibitem{boca-popa-zaharescu1}
F.~P.~Boca, V.~Pa\c{s}ol, A.~A.~Popa, A.~Zaharescu,
\emph{Pair correlation of angles between reciprocal geodesics on the modular surface},
Algebra \& Number Theory \textbf{8} (2014), no.~4, 999--1035.

\bibitem{boca-popa-zaharescu2}
F.~P.~Boca, A.~A.~Popa, A.~Zaharescu,
\emph{Pair correlation of hyperbolic lattice angles},
Int.~J.~Number Theory \textbf{10} (2014), no.~8, 1955--1989.

\bibitem{chamizo_applications_1996}
F.~Chamizo,
\emph{Some applications of large sieve in Riemann surfaces},
Acta Arith.~\textbf{77} (1996), no.~4, 315--337.

\bibitem{chatzakos_distribution_2020}
D.~Chatzakos, P.~Kurlberg, S.~Lester, I.~Wigman,
\emph{On the distribution of lattice points on hyperbolic circles},
Algebra \& Number Theory \textbf{15} (2021), no.~9, 2357--2380.

\bibitem{erdos_angular_1999}
P.~Erd\H{o}s, R.~R.~Hall,
\emph{On the angular distribution of Gaussian integers with fixed norm},
Discrete Math.~\textbf{200} (1999), no.~1--3, 87--94.

\bibitem{friedlander_opera_2010}
J.~Friedlander, H.~Iwaniec,
\emph{Opera de cribro},
American Mathematical Society Colloquium Publications, 57. American Mathematical Society, Providence, RI, 2010.

\bibitem{friedlander_hyperbolic_2009}
J.~Friedlander, H.~Iwaniec,
\emph{Hyperbolic prime number theorem},
Acta Math.~\textbf{202} (2009), no.~1, 1--19.

\bibitem{halberstam_sieve_1974}
H.~Halberstam, H.~-E.~Richert,
\emph{Sieve methods},
London Mathematical Society Monographs, 4. Academic Press [Harcourt Brace Jovanovich, Publishers], London-New York, 1974.

\bibitem{hooley_intervals_1974}
C.~Hooley,
\emph{On the intervals between numbers that are sums of two squares.~III},
J.~Reine Angew.~Math.~\textbf{267} (1974), 207--218.

\bibitem{indlekofer_scharfe_1974}
K.~-H.~Indlekofer,
\emph{Scharfe untere Absch\"atzung f\"ur die Anzahlfunktion der B-Zwillinge},
Acta Arith.~\textbf{26} (1974), 207--212.

\bibitem{ireland_classical_1990}
K.~Ireland, M.~Rosen,
\emph{A classical introduction to modern number theory},
Revised edition of Elements of number theory. Graduate Texts in Mathematics, 84. Springer-Verlag, New York-Berlin, 1982.

\bibitem{iwaniec_topics_1997}
H.~Iwaniec,
\emph{Topics in classical automorphic forms},
 Graduate Studies in Mathematics, 17. American Mathematical Society, Providence, RI, 1997.

\bibitem{iwaniec_methods_2002}
H.~Iwaniec,
\emph{Spectral methods of automorphic forms},
Second edition. Graduate Studies in Mathematics, 53. American Mathematical Society, Providence, RI; Revista Matem\'{a}tica Iberoamericana, Madrid, 2002.

\bibitem{iwaniec_analytic_2004}
H.~Iwaniec, E.~Kowalski,
\emph{Analytic number theory},
American Mathematical Society Colloquium Publications, 53. American Mathematical Society, Providence, RI, 2004.

\bibitem{kowalski_errata}
H.~Iwaniec, E.~Kowalski,
\emph{Corrections for the book Analytic Number Theory (Coll.~Publ.~53)},
2010.

\bibitem{katok_fuchsian_1992}
S.~Katok,
\emph{Fuchsian groups},
Chicago Lectures in Mathematics. University of Chicago Press, Chicago, IL, 1992. 

\bibitem{kelmer-kontorovich}
D.~Kelmer, A.~Kontorovich,
\emph{On the pair correlation density for hyperbolic angles},
Duke Math.~J.~\textbf{164} (2015), no.~3, 473--509.

\bibitem{lax_asymptotic_1982}
P.~D.~Lax, R.~S.~Phillips,
\emph{The asymptotic distribution of lattice points in Euclidean and non-Euclidean spaces},
J.~Functional Analysis \textbf{46} (1982), no.~3, 280--350.

\bibitem{malcolm}
C.~Malcolm,
\emph{The Hyperbolic Prime Number Theorem},
MsC Project, 2018.

\bibitem{marklof-vinogradov}
J.~Marklof, I.~Vinogradov,
\emph{Directions in hyperbolic lattices},
J.~Reine Angew.~Math.~\textbf{740} (2018), 161--186.

\bibitem{montgomery_ten_1994}
H.~Montgomery,
\emph{Ten Lectures on the Interface between Analytic Number Theory and Harmonic Analysis},
CBMS Regional Conference Series in Mathematics, 84. Published for the Conference Board of the Mathematical Sciences, Washington, DC; by the American Mathematical Society, Providence, RI, 1994.

\bibitem{nair_short_1998}
M.~Nair, G.~Tenenbaum,
\emph{Short sums of certain arithmetic functions},
Acta Math.~\textbf{180} (1998), no.~1, 119--144.

\bibitem{nicholls}
P.~Nicholls,
\emph{A lattice point problem in hyperbolic space},
Michigan Math.~J.~\textbf{30} (1983), no.~3, 273--287.

\bibitem{nowak_distribution_2005}
G.~W.~Nowak,
\emph{On the distribution of $M$-tuples of $B$-numbers},
Publications de l'Institut Mathematique \textbf{77} (2005), no.~91, 71--78.

\bibitem{odoni_norms_1975}
R.~W.~K.~Odoni,
\emph{On the norms of algebraic integers},
Mathematika \textbf{22} (1975), no.~1, 71--80.

\bibitem{petridis_averaging_2018}
Y.~N.~Petridis, M.~S.~Risager,
\emph{Averaging over Heegner points in the hyperbolic circle problem},
Int.~Math.~Res.~Not.~(2018), no.~16, 4942--4968.

\bibitem{risager-sodergren}
M.~S.~Risager, A.~S\"{o}dergren,
\emph{Angles in hyperbolic lattices: the pair correlation density},
Trans.~Amer.~Math.~Soc.~\textbf{369} (2017), no.~4, 2807--2841.

\bibitem{risager-truelsen}
M.~S.~Risager, J.~L.~Truelsen,
\emph{Distribution of angles in hyperbolic lattices},
Q.~J.~Math.~\textbf{61} (2010), no.~1, 117--133.

\bibitem{selberg_equidistribution_1970}
A.~Selberg,
\emph{Equidistribution in discrete groups and the spectral theory of automorphic forms},
Unpublished lecture notes, available online at https://publications.ias.edu/sites/default/files/DOCequi3.pdf, 1970.

\bibitem{steeples}
T.~Steeples,
\emph{The Hyperbolic Lattice Counting Problem},
MSc Project (2016).

\end{thebibliography}
\end{document}